\documentclass[11pt,twoside]{amsart}
\usepackage{t1enc}
\usepackage[latin2]{inputenc}

\usepackage{phaistos}

\usepackage{amsmath, amsthm, amscd, amssymb, txfonts, wasysym, verbatim}
\usepackage[T1]{fontenc}
\usepackage{graphicx, nicefrac, url, color}

\usepackage[small,nohug,
]{diagrams}
\diagramstyle[labelstyle=\scriptstyle]

\newarrow{Corr} <--->
\newarrow{Dashto} {}{dash}{}{dash}>
\newarrow{Dotsto} ....>
\newarrow{Implies} ===={=>}

\usepackage{mathpazo} 

\newcommand{\mc}{\mathcal}
\newcommand{\mbb}{\mathbb}
\newcommand{\mrm}{\mathrm}

\newcommand{\mbf}{\mathbf}

\newcommand{\0}{\emptyset}
\newcommand{\al}{\alpha}
\newcommand{\be}{\beta}
\newcommand{\ga}{\gamma}

\newcommand{\lam}{\lambda}
\newcommand{\eps}{\varepsilon}
\newcommand{\de}{\delta}
\newcommand{\om}{\omega}
\newcommand{\Om}{\Omega}

\newcommand{\la}{\langle}
\newcommand{\ra}{\rangle}

\newcommand{\NN}{\mathbb{N}}

\newcommand{\clrest}{\!\upharpoonright\!}
\newcommand{\wt}{\widetilde}
\newcommand{\dom}{\mathrm{dom}}

\newcommand{\bc}{\begin{center}}
\newcommand{\ec}{\end{center}}

\newtheorem*{nnclaim}{Claim}

\newtheorem{thm}{Theorem}[section]
\newtheorem{lem}[thm]{Lemma}
\newtheorem{prop}[thm]{Proposition}
\newtheorem{cor}[thm]{Corollary}
\newtheorem{fact}[thm]{Fact}

\theoremstyle{definition}
\newtheorem{que}[thm]{Question}

\newtheorem{df}[thm]{Definition}
\newtheorem{exa}[thm]{Example}

\newtheorem{rem}[thm]{Remark}
\newtheorem{notation}[thm]{Notation}

\title{The Zoo of combinatorial Banach spaces}

\author[P. Borodulin--Nadzieja]{Piotr Borodulin--Nadzieja}
\address[Piotr Borodulin-Nadzieja]{Mathematical Institute, University of Wroc\l aw, Pl. Grunwaldzki 2, 50-384 Wroc\l aw, Poland}
\email{pborod@math.uni.wroc.pl}

\author[B. Farkas]{Barnab\'as Farkas}
\address[Barnab\'as Farkas]{DMG/Algebra, TU Wien, Wiedner Hauptstrasse 8-10/104, 1040 Vienna, Austria}
\email{barnabasfarkas@gmail.com}

\author[S. Jachimek]{Sebastian Jachimek}
\address[Sebastian Jachimek]{Mathematical Institute, University of Wroc\l aw, Pl. Grunwaldzki 2, 50-384 Wroc\l aw, Poland}
\email{sebastian.jachimek@math.uni.wroc.pl}

\author[A. Pelczar--Barwacz]{Anna Pelczar--Barwacz}
\address[Anna Pelczar-Barwacz]{Institute of Mathematics, Faculty of Mathematics and Computer Science, Jagiellonian University, \L ojasiewicza 6, 30-348 Krak\'ow, Poland}
\email{anna.pelczar@uj.edu.pl}

\thanks{B. Farkas was supported by the Austrian Science Fund (FWF) project no. I 5918. P. Borodulin-Nadzieja and S. Jachimek were supported by the 
 Polish National Science Center under the Weave-UNISONO call in the
Weave programme, no. 2021/03/Y/ST1/00124. A. Pelczar-Barwacz was supported by the grant of the National Science Centre, Poland, no. 2020/39/B/ST1/01042}


\begin{document}

\maketitle

\begin{abstract} We study Banach spaces induced by families of finite sets in the most natural (Schreier-like) way, that is, we consider the completion $X_\mc{F}$ of $c_{00}$ with respect to the norm $\sup\{\sum_{k\in F}|x(k)|:F\in\mc{F}\}$ where $\mc{F}$ is an arbitrary (not necessarily compact) family of finite sets covering $\mbb{N}$. 

Among other results, we discuss the following: 

(1) Structure theorems bonding the combinatorics of $\mc{F}$ and the geometry of $X_\mc{F}$ including possible characterizations and variants of the Schur property, $\ell_1$-saturation, and the lack of copies of $c_0$ in $X_\mc{F}$. 

	(2) A plethora of examples including a relatively simple $\ell_1$-saturated combinatorial space which does not satisfy the Schur property, as well as a new presentation of Pe\l czy\'nski's universal space. 

(3) The complexity of the family $\{H\subseteq\NN:X_{\mc{F}\upharpoonright H}$ does not contain $c_0\}$.
\end{abstract}

\section{Introduction}
Given a family $\mc{F}$ of finite sets covering $\mbb{N}=\{1,2,\dots\}$ (or any countable set), we define the \emph{extended norm} $\|\bullet\|_\mc{F}$ on $\mbb{R}^\mbb{N}$ and the space $X_\mc{F}$ as follows (see \cite{BNF20}):
\begin{align*} \|x\|_\mc{F}=&\sup\bigg\{\sum_{k\in F}|x(k)|:F\in\mc{F}\bigg\},\\
X_\mc{F}=&\big\{x\in\mbb{R}^\mbb{N}:\|P_{[n,\infty)}(x)\|_\mc{F}\xrightarrow{n\to\infty}0\big\},
\end{align*}
where $P_A\colon \mbb{R}^\mbb{N}\to\mbb{R}^\mbb{N}$ stands for the usual coordinate projection along the set $A\subseteq\mbb{N}$. Then $X_\mc{F}$ equipped with $\|\bullet\|_\mc{F}$ is a Banach space, the completion of $c_{00}$, and the canonical
algebraic basis $(e_n)$ of $c_{00}$ is a normalized $1$-unconditional basis in $X_\mc{F}$. 

The term ''combinatorial Banach space`` was coined by Gowers in \cite{Gowers-blog} to describe such spaces for $\mc{F}$ being compact, hereditary, and spreading. The combinatorial Banach spaces understood in this way were 
studied extensively by many authors (see e.g. \cite{Argyros98}, \cite{Jordi-Stevo}, \cite{Gasparis-Leung}, \cite{ExtremeKevin}). In this paper we study spaces of the form $X_\mc{F}$ in the general setting, assuming only that $\mc{F}$ consists of finite sets and covers $\mbb{N}$. We will call all such spaces \emph{combinatorial (Banach) spaces}.

\smallskip
For example, if $\mc{F}=[\NN]^{\leq 1}=\{F\subseteq\NN:|F|\leq 1\}$ then $X_\mc{F}=c_0$, if $\mc{F}=[\NN]^{<\infty}=\{F\subseteq\NN:F$ is finite$\}$ then $X_\mc{F}=\ell_1$, and, in some sense, all combinatorial spaces are simple amalgamations of $c_0$ and $\ell_1$. However, as we will see, there are many easily definable but quite interesting, sometimes really surprising examples in between the two classical sequence spaces.  The first non-trivial application (see \cite{Sch30})  of this construction was the \emph{Schreier space} $X_\mc{S}$ generated by the \emph{Schreier family} 
\[ \mc{S}=\{\0\}\cup \big\{F\subseteq \mbb{N}:|F|\leq\min(F)\big\}.\]
The space $X_\mc{S}$ was the first example of a Banach space without the weak Banach-Saks property. Later many variants of $X_\mc{S}$ were used in various constructions of peculiar Banach
spaces (see e.g. \cite{Alspach-Argyros}, \cite{Castillo93}, \cite{Castillo-Gonzales}). These results led to the study of combinatorial spaces generated by compact families of finite sets (see e.g. \cite{Jordi15}, \cite{Jordi-Stevo}), and more
recently, motivated by set-theoretic considerations, to investigations of combinatorial spaces generated by arbitrary families of finite sets (see \cite{BNF20}). For example, a well-known non-compact example is the family 
\[ \mc{A}=\big\{F\subseteq 2^{<\mbb{N}}:F\,\text{ is a finite antichain}\big\}.\] 
The space $X_\mathcal{A}$ is a peculiar alloy of $c_0$ and $\ell_1$: Copies of $c_0$ live on the branches and copies of $\ell_1$ live on the infinite antichains. This space is called the \emph{(dyadic) stopping time space} (see e.g. \cite{BangOdell}) and H. Rosenthal proved that it contains copies of all $\ell_p$ spaces for $1\leq p<\infty$ (see \cite[Section 6]{BangOdell} and \cite[Secton 7.6]{Dew}); quite interesting for a simple amalgamation of $c_0$ and $\ell_1$.  This space was also studied in \cite{BNF20} as a Banach space analog of the so-called \emph{trace of null ideal}. 

\smallskip 
It turned out that there is a deep interaction between the canonical structures involved in these constructions: The generating family $\mc{F}$, the space $X_\mc{F}$, and the collection of \emph{$\mc{F}$-ideals}, that is, families of the form 
\[ \mc{I}_{\mc{F},\tau}=\big\{A\subseteq\mbb{N}:P_A(\tau)\in X_\mc{F}\big\}\subseteq\mc{P}(\NN)\]
where $\mc{F}$ is as above and $\tau=(\tau_n)\in (0,\infty)^\mbb{N}$ is a ``weight sequence''. The family $\mc{I}_{\mc{F},\tau}$ is indeed an \emph{ideal} on $\mbb{N}$ (in the Boolean sense), that is, it covers $\mbb{N}$, it is hereditary, and is closed
under taking unions of finitely many elements from it. Ideals of the form $\mc{I}_{\mc{F},\tau}$ belong to a specific and well-studied class, namely, they all are \emph{non-pathological analytic P-ideals} (see \cite{BNF20}). As in this paper we focus on the combinatorics of $\mc{F}$ and the geometry of $X_\mc{F}$, no detailed introduction to $\mc{F}$-ideals will be presented. Nevertheless, the spirit of analytic P-ideals is present in this paper as well, it manifests itself
in the definitions of some families $\mc{F}$. Let us mention here a flagship example (see \cite[Theorem 6.3]{BNF20}) of interactions between the (three) structures mentioned above:

\begin{thm}\label{noell1} 
Let $\mc{F}$ be a hereditary family of finite sets covering $\mbb{N}$. Then the following are equivalent:\footnote{Most implications follow from classical results, the new and difficult addition in \cite{BNF20} was (i)$\to$(iv).}
\begin{itemize}\setlength\itemsep{0.1cm}
\item[(i)] $\mc{F}$ is compact in $\mc{P}(\mbb{N})$.
\item[(ii)] $X_\mc{F}$ does not contain $\ell_1$.
\item[(iii)] $X_\mc{F}$ is $c_0$-saturated.
\item[(iv)] Non-trivial $\mc{F}$-ideals are not $F_\sigma$.
\end{itemize}
\end{thm}

In general, the purpose of our investigations is two-fold, we are looking for (1) structural results similar to the one above in the non-compact case, and (2) new constructions of Banach spaces of the form $X_\mc{F}$. Aim (2) is completed to a much bigger extent than (1), hence the title of our article, mostly because while trying to understand the structure of these spaces, we have found counterexamples to many ``conjectures'' we formulated. We are going to illustrate the wide variety of combinatorial spaces and introduce further research directions raised by these examples and their properties. We will concentrate mainly on combinatorial spaces which are closer to the $\ell_1$ extreme of the spectrum. In the light of Theorem \ref{noell1} it is natural to ask if there are similar characterizations of not containing $c_0$ or $\ell_1$-saturation in the realm of combinatorial Banach spaces. First and foremost, it is natural
to consider the three classical related properties,
\[ \text{Schur property }\longrightarrow\,\ell_1\text{-saturation }\longrightarrow\text{ no copies of }c_0, \]
and to ask if they are equivalent in the class of combinatorial spaces (recall that Theorem \ref{noell1} says that the properties $c_0$-saturation and no copies of $\ell_1$ coincide in this class). We will show that none of the above implications can be reversed in this class. These and the aforementioned (counter)examples, apart from motivating many new constructions of families and spaces, indicate that there may be no hope for satisfying analogs of Theorem \ref{noell1} (neither in the case of $\ell_1$-saturation nor in the case of no copies of $c_0$). 

\smallskip
In Section \ref{combBS} we recall some basic properties of combinatorial spaces and describe how exactly the extreme points of the dual ball $B(X^*_\mc{F})$ look like (this is a generalization of a theorem for the compact case, announced by Gowers in \cite{Gowers-blog} and proved in \cite[Theorem 4.5]{ExtremeKevin}).

\smallskip
In Section \ref{towards} we discuss characterizations and possible variants of the Schur property, $\ell_1$-saturation, and the lack of copies of $c_0$. 

\smallskip
In Section \ref{walking} (the main section of the paper) we take a walk in the Zoo of combinatorial Banach spaces. We demonstrate the rich variety of these spaces, and, among other results, we show the following:
\begin{itemize}\setlength\itemsep{0.1cm}
\item[(a)] As $c_0$ has a, up to equivalence, unique normalized unconditional basis (see \cite{LIND-equiv}), $X_\mc{F}$ is isomorphic to $c_0$ iff $(e_n)$ in $X_\mc{F}$ is equivalent to the canonical basis of $c_0$, and, obviously, this holds iff $\sup\{|F|:F\in\mc{F}\}<\infty$. Unlike in the case of $c_0$, it seems rather difficult to give a similar combinatorial characterization of $X_\mc{F}$ being isomorphic to $\ell_1$ (see Example \ref{exa-balazs}).
\item[(b)] Regarding natural dualizations of Theorem \ref{noell1}, ``nice'' properties of $\mc{F}$ do not seem to work. More precisely, one may suspect that nowhere compactness or, at least, everywhere perfectness (see Section \ref{towards}) kills copies of $c_0$. We will show that this is not the case (see Example \ref{exa-spoiled-farah} and Example \ref{iFh-n2}). 
\item[(c)] The natural ``uniform'' version of the Schur property (see Section \ref{towards}) is rather weak, we show an example of a space $X_\mc{F}$ satisfying this property but containing more or less obvious copies of $c_0$ (see Example \ref{exa-S1S2}).   
\item[(d)] As we know, in general, $\ell_1$-saturation does not imply the Schur property, the first example was constructed by Bourgain, and later more and more such examples appeared in the literature (see e.g. \cite{Hagler}, \cite{Popov}, \cite{Galego}). We add an interesting combinatorial space to this list (see Subsection \ref{rapid-farah}) which may be one of the simplest one so far (the proof as well). 
\item[(e)] We present an interesting example of a combinatorial space that does not contain $c_0$ but is not $\ell_1$-saturated either. The idea behind this construction is that for every Banach space $Y$ with unconditional basis $(b_n)$, there is a canonical family $\mc{F}$ such that a complemented block basic sequence in $X_\mc{F}$ is equivalent to $(b_n)$ (see Theorem \ref{prere}). It turns out that with $Y=\ell_2$ this canonical space does not contain $c_0$ (see Example \ref{exa-l2}).
\item[(f)] There is a universal family, that is, a hereditary cover $\mc{P}$ of $\NN$ by finite sets such that every such family $\mc{G}$ is isomorphic (via an injection) of a restriction $\mc{P}\clrest H=\{F\in\mc{P}:F\subseteq H\}$ of $\mc{P}$ (see
	Proposition \ref{fraise}). It follows that $X_\mc{P}$ contains isometric complemented copies of all combinatorial spaces (in particular, it is universal in this class). Applying the construction from (e), it follows that the space $X_\mc{P}$ is
		isomorphic to Pe\l czy\'nski's universal space (see \cite{universalPelczynski}), that is, it is a space with unconditional basis which contains complemented copies of all such spaces. However, our example is not permutatively equivalent to
		the original Pe\l czy\'nski space and so it witnesses the negative answer to an old question posed by Pe\l czy\'nski (\cite[Problem 4]{universalPelczynski}).
\end{itemize}

In Section \ref{J ideals} we study a certain combinatorial measurement of how far $X_\mc{F}$ from $c_0$ actually is, more precisely, the families of the form
\[ \mc{H}(\mc{F},\neg c_0)=\big\{H\subseteq \NN:X_{\mc{F}\upharpoonright H}\,\text{ does not contain copies of }\,c_0\big\}\]
where $X_{\mc{F}\upharpoonright H}=[(e_n)_{n\in H}]=\overline{\mrm{span}}(\{e_n:n\in H\})\subseteq X_\mc{F}$. We show that $\mc{H}(\mc{F},\neg c_0)$ is a coanalytic ideal on $\NN$ and present some interesting examples including a complete coanalytic one.

\smallskip
In Section \ref{quesec} further research directions, motivated by the results from the previous sections, are outlined.

\subsection*{Acknowledgements} We would like to thank Bal\'azs Keszegh for drawing our attention to the hypergraph used in Example \ref{exa-balazs}, and Jordi Lopez-Abad for the stimulating discussions over the subject of this paper. The first three
authors are particularly indebted to Kevin Beanland who taught them a lot about combinatorics in Banach spaces, and directly inspired several results of this article including Example \ref{exa-S1S2} and Proposition \ref{A-no-c0}.

\section{Combinatorial Banach spaces}\label{combBS} 

In this section, we will set the stage for combinatorial Banach spaces. Also, we will describe how the extreme points of their dual unit balls look like.

\smallskip
In general, we can define combinatorial spaces generated by any family $\mc{C}\subseteq\mc{P}(\mbb{N})$ covering $\mbb{N}$, also, the largest possible sequence space containing the combinatorial space (see \cite{BNF20}): 
\begin{align*}
\|x\|_\mc{C}&=\sup\bigg\{\sum_{k\in C}|x(k)|:C\in\mc{C}\bigg\},\\
\mrm{FIN}(\mc{C})&=\big\{x\in\mbb{R}^\mbb{N}:\|x\|_\mc{C}<\infty\big\},\\
\mrm{EXH}(\mc{C})&=\Big\{x\in\mbb{R}^\mbb{N}:\|P_{[n,\infty)}(x)\|_\mc{C}\xrightarrow{n\to\infty} 0\Big\}=X_\mc{C}.
\end{align*}

Then $\mrm{FIN}(\mc{C})$ and $\mrm{EXH}(\mc{C})$ equipped with $\|\bullet\|_\mc{C}$ are Banach spaces, $\mrm{EXH}(\mc{C})$ is the completion of $c_{00}$ and the canonical algebraic basis of $c_{00}$ is a normalized $1$-unconditional basis in
$\mrm{EXH}(\mc{C})$ (see e.g. \cite{BNF20}). When working with a cover $\mc{C}\subseteq\mc{P}(\Om)$ of a countable (perhaps even finite) set $\Omega$ instead of $\mbb{N}$, we define $\|\bullet\|_\mc{C}$ and $\mrm{FIN}(\mc{C})$ as above, and  
\[ X_\mc{C}=\big\{x\in\mbb{R}^\Omega:\inf\big\{\|P_{\Omega\setminus E}(x)\|_\mc{C}:E\subseteq\Omega\,\text{ is finite}\big\}=0\big\}.\]

In the general case, we will work over $\Omega=\NN$ but some of our more specific examples live on other countable underlying sets, e.g. on the binary tree 
\[2^{<\mbb{N}}=\big\{s:s=\0\,\text{ or }\,s\,\text{ is a }\,\{1,2,\dots,n\}\to\{0,1\}\,\text{function}\,(n\in\NN)\big\}\] or on $\NN^{<\NN}$. Given
a countable set $\Omega$ and $n\in\NN$, we write $[\Omega]^{\leq n}=\{E\subseteq\Omega:|E|\leq n\}$, $[\Omega]^{<\infty}=\{$finite subsets of $\Omega\}$, and $[\Omega]^\infty=\{$infinite subsets of $\Omega\}$. When referring to topological and
measure theoretic properties of families contained in  $\mc{P}(\Omega)$ for an infinite $\Omega$, we consider $\mc{P}(\Omega)\simeq 2^\Omega\simeq 2^\mbb{N}$ equipped with the usual product topology and probability measure. 

\subsection{Relevant generating families}\label{relevant} 
For a $\mc{C}\subseteq\mc{P}(\mbb{N})$ let 
\[ \mrm{hc}(\mc{C})=\big\{E\subseteq\mbb{N}:E\subseteq C\,\text{ for some }\,C\in\mc{C}\big\}\] stand for its hereditary closure and $\overline{\mc{C}}$ for its topological closure (in $\mc{P}(\NN)$, see above). One can easily show that 
$\mrm{hc}(\overline{\mc{C}})=\overline{\mrm{hc}(\mc{C})}$, in particular, if $\mc{C}$ is closed then so is $\mrm{hc}(\mc{C})$, and if $\mc{C}$ is hereditary then so is $\overline{\mc{C}}$. Also, it is trivial to see that if $\mc{C}$ covers $\mbb{N}$ then the following families generate the same norms: 
\[ \mc{C},\;\mrm{hc}(\mc{C}),\;\overline{\mc{C}},\;\mrm{hc}(\overline{\mc{C}}),\,\text{ and }\,\mrm{hc}(\mc{C})\cap [\mbb{N}]^{<\infty}=\mrm{hc}(\overline{\mc{C}})\cap[\mbb{N}]^{<\infty}.\] 
This observation basically says that it is enough to study families from 
\[ \mrm{FHC}=\big\{\mc{F}\subseteq [\mbb{N}]^{<\infty}:\mc{F}\,\text{ is hereditary and covers }\,\mbb{N}\big\},\] 
or, alternatively, from 
\[ \mrm{ADQ}=\big\{\mc{C}\subseteq\mc{P}(\NN):\mc{C}\,\text{ is compact, hereditary, and covers }\,\mbb{N}\big\},\]
elements of $\mrm{ADQ}$ are called \emph{adequate} (see \cite{adequate}). 
Considering $\mrm{FHC}\subseteq \mc{P}([\NN]^{<\infty})\simeq 2^{[\NN]^{<\infty}}\simeq 2^\NN$ and $\mrm{ADQ}\subseteq K(\mc{P}(\NN))=\{$nonempty compact subsets of $\mc{P}(\NN)\}$ equipped with the Vietoris topology, one can easily check that the map 
\[ \mrm{FHC}\to\mrm{ADQ},\,\mc{F}\mapsto\overline{\mc{F}}=\mc{F}\cup\big\{A\in [\NN]^\infty:[A]^{<\infty}\subseteq\mc{F}\big\}\] is a homeomorphism (with inverse $\mc{C}\mapsto \mc{C}\cap [\NN]^{<\infty}$) and that these spaces are homeomorphic to $2^\NN$. Mostly, we will use $\mrm{FHC}$ but sometimes it is more natural to define an adequate family.

\smallskip
Regarding compact families of finite sets, it follows that given an $\mc{F}\in\mrm{FHC}$, $\mc{F}$ is compact (i.e. $\mc{F}=\overline{\mc{F}}$) iff $\mc{F}$ does not contain infinite $\subseteq$-chains iff $\overline{\mc{F}}\subseteq[\mbb{N}]^{<\infty}$. Also, one can easily check that such an $\mc{F}$ is compact iff every sequence $(F_n)_{n\in\NN}$ in $\mc{F}$ has a subsequence $(F_n)_{n\in I}$ (for an infinite $I\subseteq\NN$) which forms a $\Delta$-system, that is, $F_n\cap F_m$ does not depend on $n,m\in I$, $n\ne m$. For detailed studies on and applications of compact families see e.g. \cite{Jordi-Stevo} or  \cite{Jordi15}. We will need an additional easy characterization:
\begin{fact}\label{scattered-fact}
Given $\mc{F}\in\mrm{FHC}$, the following are equivalent:
\begin{itemize}\setlength\itemsep{0.1cm}
\item[(i)] $\mc{F}$ is compact. 
\item[(ii)] $\mc{F}$ is scattered (that is, every non-empty $\mc{H}\subseteq\mc{F}$ contains an isolated point). 
\item[(iii)] $\overline{\mc{F}}$ is scattered.  
\end{itemize}
\end{fact}
\begin{proof}
(i)$\to$(ii): Given $\mc{H}\subseteq \mc{F}=\overline{\mc{F}}$, every $H\in\mc{H}$ can be covered by a $\subseteq$-maximal $H'\in\mc{H}$ (otherwise there would be an infinite chain in $\mc{H}\subseteq\mc{F}$), and these maximal elements are isolated points of $\mc{H}$. (ii)$\to$(iii): If there was an infinite $A\in\overline{\mc{F}}$, then $[A]^{<\infty}\subseteq\mc{F}$ would have no isolated points;  therefore, $\overline{\mc{F}}=\mc{F}$. (iii)$\to$(i) follows like (ii)$\to$(iii).   
\end{proof}

\subsection{Extreme points of $B(X^*_\mc{F})$}\label{extreme} 

In \cite[Theorem 4.5]{ExtremeKevin} the authors gave a combinatorial characterization of the extreme points of $B(X^*_\mc{F})$ for compact $\mc{F}\in\mrm{FHC}$. We will show that the same characterization works in the general setting.

\smallskip
If $X$ has a basis $(e_n)$ and $f\in X^*$, then $f(x)=\sum_{n=1}^\infty f(e_n)e^*_n(x)$ for every $x\in X$, and hence we may and will identify $f$ and $(f(e_n))\in\mbb{R}^\mbb{N}$ and consider $X^*\subseteq\mbb{R}^\mbb{N}$. 

\begin{rem} Of course, the series $\sum_{n=1}^\infty f(e_n)e^*_n$ is always weak$^*$-convergent in $X^*$. Applying Theorem \ref{noell1}, $\mc{F}\in\mrm{FHC}$ is compact iff $(e_n)$ is a shrinking basis in $X_\mc{F}$ iff $\sum_{n=1}^\infty f(e_n)e^*_n$ is norm-convergent in $X^*_\mc{F}$ for every $f\in X^*_\mc{F}$. 
\end{rem}

\begin{notation}
To avoid confusion when working with sequences of vectors and scalars, from now on, if $x,y_n,\al\in \mbb{R}^\mbb{N}$ (e.g. $x,y_n\in X_\mc{F}\subseteq \mbb{R}^\mbb{N}$ and $\al\in X^*_\mc{F}\subseteq\mbb{R}^\mbb{N}$) then we will refer to their $k$th coordinate as $x(k),y_n(k),\al(k)$ instead of using subscripts. 
\end{notation}

If $f\in X^*_\mc{F}$ and $\al=(f(e_n))\in\mbb{R}^\mbb{N}$, then we will write $f(x)=\la\al,x\ra$. It follows that if $\al,\be\in\mbb{R}^\mbb{N}$, $\al\in X^*_\mc{F}$, and $|\be(n)|\leq |\al(n)|$ for every $n$, then $\be\in X^*_\mc{F}$ as well and
$\|\be\|^*_\mc{F}\leq \|\al\|^*_\mc{F}$ where $\|\bullet\|^*_\mc{F}$ stands for the usual norm of $X^*_\mc{F}$. (As a side remark, let us mention that the norm $\|\bullet\|^*_\mc{F}$ is a Banach envelope of a certain simply definable quasi-norm of combinatorial nature, see \cite{quasiJachimek}.)     

\smallskip
If $\be\in B(X^*_\mc{F})$ then $|\be(n)|\leq 1$ for every $n$; conversely, if $\mrm{supp}(\be)\in\overline{\mc{F}}$ and $|\be(n)|\leq 1$ for every $n$, then $\be\in B(X^*_\mc{F})$. In particular, if $\sigma\in \{\pm 1,0\}^\mbb{N}$ is such that $\mrm{supp}(\sigma)\in\overline{\mc{F}}$, then $\sigma\in B(X^*_\mc{F})$ and, unless $\mrm{supp}(\sigma)=\0$, $\|\sigma\|^*_\mc{F}=1$. Now, if $\mc{H}\subseteq\overline{\mc{F}}$ define
\[ W(\mc{H})=\big\{\sigma\in\{\pm 1,0\}^\mbb{N}:\mrm{supp}(\sigma)\in \mc{H}\big\}\subseteq B(X^*_\mc{F}).\]

Notice that if $\mc{F}\subseteq\mc{H}$ then $W(\mc{H})$ is a \emph{norming set}, that is, 
\[ \|x\|_\mc{F}=\sup\big\{|\la\sigma,x\ra|:\sigma\in W(\mc{H})\big\}\,\text{ for every }\,x\in X_\mc{F},\]
equivalently (as $W(\mc{H})=-W(\mc{H})$, see \cite[Lemma 4]{Brech-isometries}), $\overline{\mrm{conv}}^{w^*}\!(W(\mc{H}))=B(X^*_\mc{F})$.

\smallskip
Also, one can easily check that the weak$^*$ topology on $W(\mc{H})$ coincides with the inherited topology from the product $\{\pm 1,0\}^\mbb{N}$; and it follows that $(W(\mc{H}), w^*)$ is compact iff $W(\mc{H})\subseteq\{\pm 1,0\}^\NN$ is compact iff $\mc{H}\subseteq\mc{P}(\mbb{N})$ is compact. It also follows that 
\[ X_\mc{F}\to C(W(\overline{\mc{F}}),w^*),\;x\mapsto\la\bullet,x\ra\,\text{ is an isometric embedding.}\]

Given $\mc{F}\in\mrm{FHC}$ let $\max(\overline{\mc{F}})$ be the set of all maximal elements of $(\overline{\mc{F}},\subseteq)$, that is, 
\[ \max(\overline{\mc{F}})=\big\{A\in\overline{\mc{F}}:A\,\text{ has no proper extension in }\overline{\mc{F}}\big\}.\] 
Then $\max(\overline{\mc{F}})$ covers $\NN$, moreover, $W(\max(\overline{\mc{F}}))$ is also a norming set.  

\begin{prop}\label{extremedual} $\mrm{Ext}(B(X_\mc{F}^*))=W(\max(\overline{\mc{F}}))$. 
\end{prop}
\begin{proof} $\mrm{Ext}(B(X_\mc{F}^*))\subseteq W(\max(\overline{\mc{F}}))$: Consider the locally convex space $(X^*_\mc{F},w^*)$. We know that $B(X^*_\mc{F})=\overline{\mrm{conv}}^{w^*}(W(\overline{\mc{F}}))$ is weak$^*$ compact, and hence, applying Milman's theorem (see \cite[Theorem 3.66]{bible}), $\mrm{Ext}(B(X^*_\mc{F}))\subseteq \overline{W(\overline{\mc{F}})}^{w^*}=W(\overline{\mc{F}})$. Now let $\sigma\in W(\overline{\mc{F}})$ such that $\mrm{supp}(\sigma)$ is not maximal in $\overline{\mc{F}}$ and fix a $k\in\mbb{N}\setminus\mrm{supp}(\sigma)$ such that $\mrm{supp}(\sigma)\cup\{k\}\in\overline{\mc{F}}$. It follows that $\sigma\ne \sigma\pm e^*_k\in B(X^*_\mc{F})$ and hence $\sigma$ is not an extreme point.

\smallskip
$W(\max(\overline{\mc{F}}))\subseteq\mrm{Ext}(B(X_\mc{F}^*))$: Let $\sigma\in W(\max(\overline{\mc{F}}))$, $\al,\be\in B(X^*_\mc{F})\setminus\{\sigma\}$, and assume on the contrary that $\sigma=(\al+\be)/2$. It follows that $\al\clrest \mrm{supp}(\sigma)=\be\clrest \mrm{supp}(\sigma)=\sigma\clrest \mrm{supp}(\sigma)$, and hence  there is a $k\in \mrm{supp}(\al)\setminus \mrm{supp}(\sigma)$. As $\mrm{supp}(\sigma)\cup\{k\}\notin\overline{\mc{F}}$, we can pick a finite non-empty $S\subseteq \mrm{supp}(\sigma)$ such that $S\cup\{k\}\notin\mc{F}$. We define $x\in X_\mc{F}$ such that $\mrm{supp}(x)=S\cup\{k\}$ and $x(n)=\mrm{sign}(\al(n))/|S|$ for every $n\in\mrm{supp}(x)$. Then $\|x\|_\mc{F}=1$ but $\la\al,x\ra=1+|\al(k)|/|S|>1$, a contradiction. 
\end{proof}

\section{Towards structural theorems}\label{towards}

Theorem \ref{noell1} characterizes the lack of isomorphic copies of $\ell_1$ and $c_0$-saturation in the realm of combinatorial
Banach spaces by bonding the combinatorics of $\mc{F}$ with the geometry of $X_\mc{F}$ (in the compact case). This section is devoted to discussing (possible) characterizations of some close relatives and ``duals'' of these properties: The Schur property, the lack of copies of $c_0$, and $\ell_1$-saturation.

\subsection{Characterization of the Schur property} 

We begin with a characterization of the Schur property in the class of combinatorial spaces. We will need the following easy lemma:

\begin{lem}\label{Sch-lemma} Given $\mc{F}\in\mrm{FHC}$, the following are equivalent:
\begin{itemize}\setlength\itemsep{0.1cm}
\item[(a)] $X_\mc{F}$ satisfies the Schur property.
\item[(b)] $\mc{F}$-supported normalized block basic sequences are not weakly null.
\end{itemize}
\end{lem}

\begin{proof} To show the non-trivial implication, assume that $(x_n)$ witnesses the failure of the Schur property, that is, $(x_n)$ is weakly null and $\|x_n\|=1$ for every $n$. Applying the Selection Principle (see \cite[Corollary C.2]{Bessaga58}), by thinning out our sequence, we can assume that $(x_n)$ is equivalent to a normalized block basic sequence $(y_n)$. 

	Then $(y_n)$ is also weakly null because given an $f\in X^*_\mc{F}$, $f\clrest [(y_n)]\in [(y_n)]^*\simeq [(x_n)]^*$, it corresponds to an $f'\in [(x_n)]^*$ which extends to an $f''\in X^*_\mc{F}$, therefore $f'(x_n)=f''(x_n)\to 0$, and hence $f(y_n)\to 0$ as well.

	Now, for each $n$, pick an $F_n\in\mc{F}$ such that $\|P_{F_n}(y_n)\|_\mc{F}=1$ and let $z_n=P_{F_n}(y_n)$, an $\mc{F}$-supported normalized block basic sequence. To finish the proof, we show that $(z_n)$ is weakly null. Assume on the contrary that there is a $\be\in X^*_\mc{F}$ such that $\la\be,z_n\ra\not\to 0$. If $A=\bigcup_{n=1}^\infty F_n$ and $\ga=P_A(\be)\in X^*_\mc{F}$, then $\la\ga,y_n\ra=\la\be,z_n\ra\not\to 0$, a contradiction.
\end{proof}

Before the next theorem, we recall a known but perhaps rarely mentioned result (see \cite[Fact 3.119 and Theorem 3.124]{bible}): 

\begin{thm}\label{suplimsup}
Given a Banach space $X$ and bounded sequence $(x_n)$ in $X$. Then 
\[\tag{$\sharp$} \sup\Big\{\limsup_{n\to\infty}f(x_n):f\in \mrm{Ext}(B(X^*))\Big\}=\sup\Big\{\limsup_{n\to\infty}g(x_n):g\in B(X^*)\Big\}.\]
\end{thm}

\begin{thm}\label{Schur} $X_\mc{F}$ satisfies the Schur property if, and only if the following holds: 
\[\tag{$S$}  \forall\text{ $\mc{F}$-supp. norm. bl. basic }\,(x_n)\,\text{ in }\,X_\mc{F}\;\exists\,A\in\overline{\mc{F}}\;\limsup_{n\to\infty}\|P_A(x_n)\|_\mc{F}>0.\] 
\end{thm}
\begin{proof}
($S$) implies the Schur property: We will apply Lemma \ref{Sch-lemma}. Fix an $\mc{F}$-supported normalized block basic sequence $(x_n)$ and an $A\in\overline{\mc{F}}$ such that $\limsup_n\|P_A(x_n)\|_\mc{F}>0$. We can assume that $A\subseteq\bigcup_{n=1}^\infty\mrm{supp}(x_n)$ and define $\eps\in\{\pm 1,0\}^\mbb{N}$ as follows: $\mrm{supp}(\eps)=A$ and $\eps(k)=\mrm{sign}(x_n(k))$ if $k\in\mrm{supp}(x_n)$. Then $\eps\in X^*_\mc{F}$ and $\la\eps,x_n\ra=\|P_A(x_n)\|_\mc{F}\not\to 0$.  

\smallskip
Conversely, assume on the contrary that $X_\mc{F}$ satisfies the Schur property and there is an $\mc{F}$-supported normalized block basic sequence $(x_n)$ such that $P_A(x_n)\to 0$ for every $A\in\overline{\mc{F}}$. If $\sigma\in \mrm{Ext}(B(X^*_\mc{F}))=W(\max(\overline{\mc{F}}))$ with $A=\mrm{supp}(\sigma)$, then $|\la\sigma, x_n\ra|\leq \|P_A(x_n)\|_\mc{F}$ and hence the left side of ($\sharp$) is simply $\sup\{0\}=0$. Applying that the right side equals $0$, $\liminf_ng(x_n)\leq\limsup_ng(x_n)\leq 0$ for every $g\in B(X^*_\mc{F})$ and $\liminf_ng(x_n)$ cannot be negative (otherwise $\limsup_n(-g(x_n))=-\liminf_n g(x_n)>0$), it follows that $(x_n)$ is weakly null, a contradiction.      
\end{proof}

How do we apply this characterization in practice? It is trivial to check that the following simple combinatorial property implies ($S$): 
\[\tag{$S^*$} \forall\,\text{ pairwise disjoint }\,(F_n)\,\text{ in }\,\mc{F}\;\exists\,\text{ finite }\,\mc{A}\subseteq \overline{\mc{F}}\;\,\exists^\infty\,n\;F_n\subseteq \bigcup\mc{A}.\]
As we will see, all our concrete examples satisfying the Schur property actually satisfy ($S^*$) because of rather obvious reasons and typically with a single $A\in\overline{\mc{F}}$ (instead of a finite $\mc{A}\subseteq\overline{\mc{F}}$). 

\smallskip
Property ($S$) is not really a property of $\mc{F}$, it is ``too geometric'', one may consider its ``uniform'' version:
\[\tag{$U_S$} \forall\,\text{ pairwise disjoint }\,(F_n)\,\text{ in }\,\mc{F}\setminus\{\0\}\;\exists\,A\in\overline{\mc{F}}\;\limsup_{n\to\infty}\frac{|A\cap F_n|}{|F_n|}>0.\]
This is indeed the special case of ($S$) applied to the sequence $x_n=\chi_{F_n}/|F_n|$. We will see that this property is quite weak, there may even be copies of $c_0$ in an $X_\mc{F}$ satisfying ($U_S$) (see Example \ref{exa-S1S2}).

\subsection{Characterization of the lack of copies of $c_0$} 

Towards possible characterizations of ``$X_\mc{F}$ does not contain $c_0$'', let us first recall some results. First, note that $X_\mc{F}=\mrm{EXH}(\mc{F})$ does not contain $c_0$ iff $(e_n)$ is boundedly complete iff
$\mrm{EXH}(\mc{F})=\mrm{FIN}(\mc{F})$ (see \cite[Theorem 5.4]{BNF20} for some further equivalent statements in this list). In this section, we are looking for characterizations of a more combinatorial nature.

\smallskip
We know that a normalized basic sequence $(x_n)$ in a Banach space $X$ is equivalent to the usual basis of $c_0$ iff 
\[ \exists\,K>0\;\forall\,n\;\forall\,a\in\mbb{R}^n\;\bigg\|\sum_{i=1}^na(i)x_i\bigg\|\leq K\max_{i=1,\dots,n}|a(i)|.\] 

Now, if $X$ has an unconditional basis $(b_n)$ and $X$ contains a copy of $c_0$, then, assuming $(b_n)$ is normalized, there is a normalized block basic (nbb) sequence $(x_n)$ w.r.t. $(b_n)$ which is equivalent to the canonical basis of $c_0$ (see \cite[Theorem 3.3.2]{AK16}). As a normalized block basic sequence in such a space is automatically unconditional, it follows that such a sequence is equivalent to the basis of $c_0$ iff   
\[\tag{nbb $\sim c_0$} \exists\,K>0\;\forall\,n\;\bigg\|\sum_{i=1}^nx_i\bigg\|\leq K.\]

If $X=X_\mc{F}$, $b_n=e_n$, and for a normalized block basic sequence $\overline{x}=(x_n)$, $s(\overline{x})$ stands for the formal sum of $(x_n)$, then (nbb $\sim c_0$) is equivalent to $\|s(\overline{x})\|_\mc{F}<\infty$. Furthermore, in this
case, we can always assume that such a normalized block basic sequence is $\mc{F}$-supported, that is, $\mrm{supp}(x_n)\in\mc{F}$ for every $n$ because if $\|P_{F_n}(x_n)\|_\mc{F}=1$ with some $F_n\in\mc{F}$ and $y_n=P_{F_n}(x_n)$, then
$\overline{y}=(y_n)$ is an $\mc{F}$-supported normalized block basic sequence and $\|s(\overline{y})\|_\mc{F}\leq \|s(\overline{x})\|_\mc{F}$.

Of course, there are other natural ways to express $\|s(\overline{x})\|_\mc{F}$: 
\[ \|s(\overline{x})\|_\mc{F}=\sup_{H\in\mc{H}}\|P_H(s(\overline{x}))\|_\mc{F}=\sup_{H\in \mc{H}}\sum_{n=1}^\infty\|P_H(x_n)\|_\mc{F}\]
where $\mc{F}\subseteq\mc{H}\subseteq\overline{\mc{F}}$ and, in this case, $\|P_H(x)\|_\mc{F}=\|P_H(x)\|_1$ for every $x\in\mbb{R}^\NN$.

\smallskip
Reformulating the above, $X_\mc{F}$ does not contain a copy of $c_0$ iff the following holds: 
\[\tag{$\neg c_0$}\forall\text{ $\mc{F}$-supp. norm. bl. basic }\,(x_n)\,\text{ in }\,X_\mc{F}\;\sup_{A\in\overline{\mc{F}}}\sum_{n=1}^\infty\|P_A(x_n)\|_\mc{F}=\infty.\]

Now, just like ($S$), the property ($\neg c_0$) also has a weak, uniform version we will further discuss below: 
\[\tag{$U_{\neg c_0}$} \forall\text{ pairwise disjoint }(F_n)\text{ in }\mc{F}\setminus\{\0\}\;\sup_{A\in\overline{\mc{F}}}\sum_{n=1}^\infty\frac{|A\cap F_n|}{|F_n|}=\infty.\]

\subsection{Nowhere compactness and everywhere perfectness}

Theorem \ref{noell1} combined with Fact \ref{scattered-fact} says, in particular, that $X_\mc{F}$ does not contain $\ell_1$ iff $X_\mc{F}$ is $c_0$-saturated iff $\mc{F}$ is compact iff $\mc{F}$ is scattered iff $\overline{\mc{F}}$ is scattered. So, it is natural (but probably naive) to ask if any of the dual properties, lack of copies of $c_0$ or $\ell_1$-saturation, may be characterized by some sort of anti-compactness or perfectness. 

\smallskip
When looking for strong negations of compactness, the very first idea is probably the following: We say that $\mc{F}$ is \emph{nowhere compact} if $\mc{F}\clrest H$ is not compact in $\mc{P}(H)$ for any infinite $H\subseteq\mbb{N}$, i.e. 
\[\tag{nw cpt} \forall\,H\in [\mbb{N}]^{\infty}\;\exists\,A\in [H]^{\infty}\;[A]^{<\infty}\subseteq\mc{F}.\]
Of course, a non-compact family $\mc{F}$ is not necessarily nowhere compact, consider e.g. 
\[\mc{F}_\text{col}=\big\{F\in [\mbb{N}\times\mbb{N}]^{<\infty}:\exists\,n\;F\subseteq \{n\}\times\mbb{N}\big\}.\]
It is obviously non-compact but $\mc{F}_\text{col}\clrest \{(n,a_n):n\in\NN\}$ is compact for every sequence $a_n\in\NN$. Now, both ($U_S$) (because of trivial reasons) and ($\neg c_0$) (because of Theorem \ref{noell1}) imply nowhere compactness but unfortunately this property is very weak, it does not imply ($U_{\neg c_0}$) (see Example \ref{iFh-n2}).

\smallskip
Regarding perfectness, first of all, it would be convenient to have the following fact available.

\begin{fact}\label{perfect} Given $\mc{F}\in\mrm{FHC}$, the following are equivalent:
\begin{itemize}\setlength\itemsep{0.1cm}
\item[(a)] $\mc{F}$ is perfect in $[\NN]^{<\infty}$.
\item[(b)] $\overline{\mathcal{F}}$ is perfect in $\mc{P}(\NN)$.
\item[(c)] Every $F\in \mathcal{F}$ can be extended to an infinite $A \in \overline{\mathcal{F}}$.
\end{itemize}
\end{fact}
\begin{proof}
(a)$\to$(b): Given an $A\in\overline{\mc{F}}\setminus \mc{F}$, we know that $A$ is infinite and $[A]^{<\infty}\subseteq \mc{F}$, and, of course, $A\in \overline{[A]^{<\infty}}$. (b)$\to$(c): Assuming there is no such $A\in\overline{\mc{F}}$, every $F'\in\max(\overline{\mc{F}})$  covering $F$ is finite, and for such an $F'$ the set $\{S\subseteq\NN:F'\subseteq S\}$ is open and witnesses that $F'$ is an isolated point of $\overline{\mc{F}}$. (c)$\to$(a): $\mc{F}=\overline{\mc{F}}\cap [\NN]^{<\infty}$ is always closed in $[\NN]^{<\infty}$. If $F\in\mc{F}$ and $A=\{a_1<a_2<\dots\}\in\overline{\mc{F}}$ covers $F$, then $F\cup \{a_n\}\in\mc{F}$ converges to $F$.
\end{proof}

This fact basically says that perfectness means that $\mathcal{F}$ is ``induced'' by a family of infinite subsets of $\mathbb{N}$, namely, by  $\overline{\mathcal{F}} \cap [\mathbb{N}]^\infty$. Just like in the case of anti-compactness, we may wish to go further and define the following: We say that $\mc{F}$ is \emph{everywhere perfect} if $\mc{F}\clrest H$ is perfect in $[H]^{<\infty}$ for every infinite $H\subseteq \NN$, i.e. $\overline{\mc{F}}\clrest H=\overline{\mc{F}\clrest H}$ is perfect in $\mc{P}(H)$ for every infinite $H\subseteq\NN$, i.e.
\[\tag{ew pft} \forall\,H\in[\NN]^{\infty}\;\forall\,F\in \mathcal{F}\clrest H\;\exists\,A\in\overline{\mc{F}}\cap [\NN]^\infty\;F\subseteq A\subseteq H.\] 
Of course, (everywhere) perfect families are (nowhere) compact but everywhere perfectness is still too weak, it still does not imply ($U_{\neg c_0}$) (see Example \ref{iFh-n2}). At the same time, somehow, when looking for examples of combinatorial spaces without copies of $c_0$, it seems natural to look among everywhere perfect families (see also Section \ref{quesec}).
 
Also, there are no further implications between (everywhere) perfectness and (nowhere) compactness: The family $\mc{F}_\text{col}$ witnesses that a perfect family can easily have compact restrictions, and if   
\[ \mc{F}_\text{oe}=\big[\{\text{odd numbers}\}\big]^{<\infty}\cup\big[\{\text{even numbers}\}\big]^{<\infty}\] 
then $\mc{F}_\text{oe}\cup \{\{1,2\}\}$ is nowhere compact (actually, it satisfies ($S^*$) as well) but $\{1,2\}$ is an isolated point. This example also shows that perfectness is somehow too specific, it is very easy to ruin simply by adding an isolated point. Notice that this manipulation is cheap in the sense that, though, $\mc{F}_\text{oe}\cup\{\{1,2\}\}$ is not perfect, it is, of course, equivalent to the everywhere perfect family $\mc{F}_\text{oe}$  (for related questions see Section \ref{quesec}). 

\begin{rem}
The notion of everywhere perfectness can be seen as an extension of  Ellentuck-perfectness from $[\mbb{N}]^\infty$ to $\mc{P}(\NN)$:  Recall that the Ellentuck topology on $[\mathbb{N}]^\infty$ is generated by the basic sets of the form 
\[ \langle E, n, H\rangle = \big\{B\in [\NN]^\infty: E \subseteq B \subseteq E\cup H\big\}\] where $E\subseteq\{1,2,\dots,n\}$ and $H\subseteq \NN\setminus\{1,\dots,n\}$ is infinite. This topology on $[\NN]^\infty$ is finer than the one inherited from the product topology on $\mc{P}(\NN)$ because if 
\[ [E,n]=\big\{B\in\mc{P}(\NN):E=\{1,2,\dots,n\}\cap B\big\}\] stands for the usual basic open set in $\mc{P}(\NN)$, $E\subseteq\{1,2,\dots,n\}$, then $[E,n]\cap [\NN]^\infty=\la E,n,(n,\infty)\ra$. 

It turns out that if $\mc{F}\in\mrm{FHC}$, then $\overline{\mc{F}}\cap [\NN]^\infty$ is either empty or Ellentuck-perfect: If $A\in \la E,n,H\ra\cap \overline{\mc{F}}$ and $k\in A\setminus E$ then $A\ne A\setminus \{k\}\in \la E,n,H\ra\cap\overline{\mc{F}}$. To extend the meaning of Ellentuck-perfectness to $\mc{P}(\NN)$ (that is, to $\mc{F}$ and to $\overline{\mc{F}}$) we shall backtrack how we obtain usual perfectness in $\mc{P}(\NN)$ from the product topology on $[\NN]^{\infty}$: A non-empty hereditary $\mc{C}\subseteq\mc{P}(\NN)$ is perfect iff 
\[ \forall\text{ basic open }\,[E,n]\;\big(E\in \mc{C}\longrightarrow [E,n]\cap [\NN]^\infty\cap \mc{C}\ne\0\big).\] 
Now, extending Ellentuck-perfectness to $\mc{P}(\NN)$ the same way would be
\[ \forall\text{ basic open }\,\la E,n,H\ra\;\big(E\in \mc{C}\longrightarrow \la E,n,H\ra\cap [\NN]^\infty\cap \mc{C}\ne\0\big),\]
and this property happens to be equivalent to everywhere perfectness. 
\end{rem}

Let us summarize the implications between the properties we have encountered in this section:
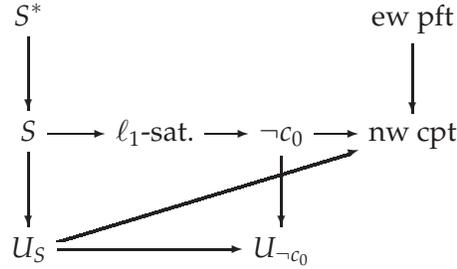
\begin{diagram}
S^*& & & & & & \text{ew pft}\\
\dTo& & & & & & \dTo \\
S & \rTo & \ell_1\text{-sat.} & \rTo & \neg c_0 & \rTo & \text{nw cpt} \\
\dTo &   & &  &  \dTo & \ruTo(6,2) \\
U_S &  & \rTo & & U_{\neg c_0}
\end{diagram}
Most properties we would like to characterize are located between ($S^*$) and ($U_{\neg c_0}$) which do not seem to be drastically far away from each other. On the other hand, apart from two questions (namely, if either $(S)\to (S^*)$ or $(U_{\neg c_0})\to$(nw cpt) hold, see Section \ref{quesec}), the examples in the next section (together with the easy ones from this section) show that in the realm of combinatorial spaces there are no further implications between these properties.

\section{Walking in the Zoo}\label{walking}

\subsection{Compact families.}\label{compact} As we have mentioned, combinatorial Banach spaces induced by compact families have been intensively studied. So, we will only briefly overview the classical examples.
The most classical one is, of course, $c_0$. It is generated by the family $[\NN]^{\leq 1}$. Notice also that $\mrm{FIN}([\NN]^{\leq 1})=\ell_\infty$.
\smallskip 

In fact, $c_0$ is an instance of the generalized Schreier spaces. 

\begin{exa}[Schreier spaces] \label{exa-Schreier}
For $\al<\om_1$ let $\mc{S}_\al$ be the $\al$th Schreier family on $\NN$, for example, $\mc{S}_0=[\NN]^{\leq 1}$, $\mc{S}_1=\mc{S}=\{\0\}\cup \{F\subseteq\NN:|F|\leq\min(F)\}$ is the classical Schreier family, 
\[ \mc{S}_2=\{\0\}\cup\bigg\{\bigcup_{i=1}^nE_i:E_i\in\mc{S}_1\setminus\{\0\}\,\text{ and }\,\{n\}<E_1<E_2<\dots<E_n\bigg\},\] 
where $E<F$ iff $\max(E)<\min(F)$ ($E,F\in [\NN]^{<\infty}\setminus\{\0\}$), etc. These families are compact (see e.g. \cite{Alspach-Argyros} for much more details).
\end{exa}

Another well-known example of a Banach space induced by a compact family is the following.

\begin{exa}[$c_0$- and $\ell_\infty$-products]\label{exa-c_0-product} Fix a partition $(V_n)$ of $\NN$ into non-empty finite sets. Let $\mc{Q}_n$ be a hereditary cover of $V_n$ and let $\mc{Q} = \bigcup_n \mc{Q}_n$. Then $\mc{Q}$ is compact and $X_\mc{Q}=\mrm{EXH}(\mc{Q})$ is (isomorphic to) the $c_0$-product of $(X_{\mc{Q}_n})$. E.g. if $\mathcal{Q}_n = \mathcal{P}(V_n)$ then $X_{\mc{Q}_n}=\ell_1(V_n)$ is $\mathbb{R}^{V_n}$ equipped with the $\ell_1$ norm. Similarly,  $\mrm{FIN}(\mc{Q})$ is the $\ell_\infty$-product of $(X_{\mc{Q}_n})$.
\end{exa}

Infinite elements in the partition above lead to the simplest non-compact families, see e.g. $\mc{F}_\text{col}$ and $\mc{F}_\text{oe}$ in the previous section.

\subsection{Creatures living on trees.} Now we will present two spaces defined by H. Rosenthal, the stopping time space $S$ and its ``separable dual'' $B$ (see e.g. \cite{BangOdell} and \cite{Dew}). Both of them are combinatorial spaces induced by families on the dyadic tree.

\begin{exa}[the space $S$]\label{exa-antichains}  Let 
	\[ \mc{A}=\big\{F\subseteq 2^{<\NN}:F\,\text{ is a finite antichain}\big\}.\] The space $S=X_{\mc{A}}$ is called the \emph{(dyadic) stopping time space}. H. Rosenthal proved that it contains copies of all $\ell_p$ spaces, $1\leq p<\infty$ (see \cite[Section 6]{BangOdell} and \cite[Secton 7.6]{Dew}). Therefore, the naive impression that combinatorial spaces (simple amalgamations of $c_0$ and $\ell_1$) are $\{c_0,\ell_1\}$-saturated, that is, all infinite dimensional closed subspaces contain copies of either $c_0$ or $\ell_1$, is false.

\smallskip
The space $X_\mc{A}$ was also studied in \cite{BNF20} (though the authors were not aware of its history) as an analog of the trace of the null ideal. For example, it was shown that $\mrm{FIN}(\mc{A})$ contains a canonical isometric copy of the Banach space $M(2^\NN)$ of finite signed Borel measures on $2^\NN$ equipped with the total variation norm: For $\mu\in M(2^\NN)$ define $x_\mu:2^{<\NN}\to\mbb{R}$, $x_\mu(t) = \mu([t])$ where $[t]=\{\eps\in 2^\NN:\eps$ extends $t\}$ is the basic clopen set generated by $t$.
	Then $x_\mu\in\mrm{FIN}(\mc{A})$ and $\| x_\mu \|_{\mc{A}}=\|\mu\|_\mrm{tv}$. 
\end{exa}

\begin{exa}[the space $B$]\label{exa-chains} Let 
\[ \mc{C}=\big\{E\subseteq 2^{<\NN}:E\,\text{ is a finite chain}\big\}.\] 
The space $B=X_{\mc{C}}$ is also well-studied, for example, we know that $X_\mc{C}$ contains isometric copies of all the Banach spaces with unconditional basis (see \cite[Theorem 2]{BangOdell}). The combinatorial symmetry between chains and antichains lifts to a symmetry between $X_\mc{C}$ and $X_\mc{A}$, namely, $X_\mc{A}^*=\mrm{FIN}(\mc{C})$ and $X_\mc{C}^*=\mrm{FIN}(\mc{A})$ (see \cite[Proposition 1]{BangOdell}). It follows that if $(a_t)$ ($t\in 2^{<\NN}$) is the canonical basis of $X_\mc{A}$ and $(c_t)$ is of $X_\mc{C}$, then $[(a^*_t)]=X_\mc{C}$ and $[(c^*_t)]=X_\mc{A}$.

\smallskip
Like in the case of the family $\mc{A}$, we can find a canonical copy of a classical Banach space, this time of $C(2^\mathbb{N})$, in $\mrm{FIN}(\mc{C})$ (see \cite[Theorem 2]{BangOdell}).
\end{exa}

The above results allow us to say, somewhat informally, that the space $B$ is the combinatorial version of $C(2^\mathbb{N})$ and $S$ is the combinatorial version of $C(2^\mathbb{N})^*\simeq M(2^\NN)$.

\subsection{Farah families}\label{farah-families} 
This class of families is motivated by the definition of an analytic P-ideal due to Farah (see \cite{Farah}). We start with a modified version of the original family he used in his example.

\begin{exa}[Farah family]\label{exa-farah} Let
\[ \mrm{Fh}=\bigg\{F\in [\NN]^{<\infty}:\big|F\cap [2^{n-1},2^n)\big|\leq \frac{2^{n-1}}{n}\,\text{ for every }\,n\geq 1\bigg\}.\] 
Then $\mrm{Fh}$ satisfies ($S^*$), it is everywhere perfect, and $X_\mrm{Fh}$ is not isomorphic to $\ell_1$ (see \cite{BNF20}).
\end{exa}

In general, given a partition $(V_n)$ of $\NN$ (or of any countable set) into non-empty finite sets and hereditary covers $\mc{Q}=(\mc{Q}_n)$, $\mc{Q}_n$ of $V_n$, we define the associated \emph{Farah family} as follows:
\[ \mrm{Fh}(\mc{Q})=\big\{F\in [\NN]^{<\infty}:F\cap V_n\in\mc{Q}_n\,\text{ for every }\,n\big\}.\]
The family $\mrm{Fh}(\mc{Q})$ also satisfies ($S^*$), it is everywhere perfect, and the space $X_{\mrm{Fh}(\mc{Q})}$ is (isomorphic to) the $\ell_1$-product of $(X_{\mc{Q}_n})$. 

\begin{rem}
Clearly, the presentation of $\mrm{Fh}(\mc{Q})$ is not unique. At the same time, there is an easy combinatorial characterization of these families, and it also provides us with a unique presentation: Given $\mc{F}\in\mrm{FHC}$, we say that a finite set $H\subseteq\NN$ is \emph{$\mc{F}$-decomposable} if there is a partition $H=H_0\cup H_1$ into non-empty sets such that $\mc{F}\clrest H=\{E_0\cup E_1:E_i\in\mc{F}\clrest H_i\}$, otherwise we say that $H$ is \emph{$\mc{F}$-indecomposable}. Let $\mrm{ID}(\mc{F})$ be the family of all $\mc{F}$-indecomposable sets and $\max(\mrm{ID}(\mc{F}))$ consist of the maximal such sets (i.e. the maximal elements in $(\mrm{ID}(\mc{F}),\subseteq)$). 

For example, $\mrm{ID}([\NN]^{\leq 1})=[\NN]^{<\infty}$ and $\max(\mrm{ID}([\NN]^{\leq 1}))=\0$; $\mrm{ID}([\NN]^{<\infty})=[\NN]^{\leq 1}$ and so 
$\max(\mrm{ID}([\NN]^{<\infty}))=[\NN]^1$; 
\[\mrm{ID}(\mc{S})=[\NN]^{\leq 1}\cup \big([\NN]^{<\infty}\setminus\mc{S}\big)\] and hence $\max(\mrm{ID}(\mc{S}))=\0$; and finally, 
\[ \mrm{ID}(\mrm{Fh})=[\NN]^{\leq 1}\cup\bigcup_{n\geq 1}\bigg\{H\subseteq I_n:|H|>\frac{2^{n-1}}{n}\bigg\}\] and so $\max(\mrm{ID}(\mrm{Fh}))=\{I_n:n\geq 1\}$. 

It is straightforward to show the following: (1) Maximal $\mc{F}$-indecomposable sets are either disjoint or coincide. (2) A family $\mc{F}$ is of the form $\mrm{Fh}(\mc{Q})$ for some $\mc{Q}$ iff $\max(\mrm{ID}(\mc{F}))$ covers $\NN$; and in this case, $\mc{F}$ is the Farah family generated by $(\mc{F}\clrest H:H\in\max(\mrm{ID}(\mc{F})))$.
\end{rem}

\begin{exa}[a strange case of $X_\mc{F}\simeq\ell_1$, Bal\'azs Keszegh]\label{exa-balazs} This example suggests that characterizing those $\mc{F}$ such that $X_\mc{F}$ is isomorphic to $\ell_1$, i.e. $(e_n)$ in $X_\mc{F}$ is equivalent to the usual basis of $\ell_1$, is pretty far from obvious. 

For the rest of this example fix $m>1$. For $k\in\NN$ define the following hypergraph on $V_{k}=[\{1,2,\dots,mk\}]^k$, that is, a collection of subsets, called (hyper)edges, of the underlying set $V_{k}$: 
\[ \mc{H}_{k}=\big\{H^k_i=\{a\in V_{k}:i\in a\}:i=1,2,\dots,mk\big\}.\]
Now, let $\mc{Q}_k$ be the hereditary closure of $\mc{H}_{k}$, $\mc{Q}=(\mc{Q}_k)$, and $V=\bigcup_k V_k$. We show that $X_{\mrm{Fh}(\mc{Q})}$ is $m$-isomorphic to $\ell_1(V)$, i.e. that $X_{\mc{Q}_k}$ is $m$-isomorphic to $\ell_1(V_k)$: If $x\in \mbb{R}^{V_k}$ then
\[ \|x\|_1=\sum_{a\in V_{k}}|x(a)|= \frac{1}{k}\sum_{i=1}^{mk}\sum_{a\in H^k_i}|x(a)|\leq \frac{1}{k}mk\|x\|_{\mc{H}_{k}}=m\|x\|_{\mc{H}_{k}}\] 
where the second equality holds because each $a\in V_{k}$ is covered by exactly $|a|=k$ many edges.

\smallskip 
Why is this interesting? Notice, that if $I\subseteq \{1,2,\dots,mk\}$ and $|I|=(m-1)k$, then $\{H^k_i:i\in I\}$ does not cover $V_{k}$ because $V_{k}\setminus\bigcup_{i\in I}H^k_i=\{\{1,2,\dots,mk\}\setminus I\}$ (it also follows that any $(m-1)k+1$ many edges cover $V_k$). If we increase $k$, then we need more and more sets from $\mc{Q}_{k}$ to cover $V_{k}$, and seemingly this yields that $\|\bullet\|_{\mc{H}_k}$ is getting further and further from $\|\bullet\|_1$; yet somehow, according to the above, it is not.
\end{exa}

\subsection{Farah families with intervals}\label{farah-families-1} Let $V_n$, $\mc{Q}_n$, and $\mc{Q}$ be like in the definition of Farah families and define
\begin{align*} \mrm{iFh}(\mc{Q})&=\big\{F\in [\NN]^{<\infty}:F\cap V_n \in \mc{Q}_n\,\text{ for all but possibly one $n$}\big\}\\
&=\big\{F\cup E:F\in \mrm{Fh}(\mc{Q})\,\text{ and }\,E\subseteq V_n\,\text{ for some }\,n\big\}.
\end{align*} 
Adding the intervals to $\mrm{Fh}(\mc{Q})$ may look like a cosmetic modification but this extension can change the resulting combinatorial space quite fundamentally; of course, $\mrm{iFh}(\mc{Q})$ is still everywhere perfect. We begin with the modification of the original Farah family $\mrm{Fh}$.

\begin{exa}[$\mrm{Fh}$ with intervals]\label{exa-spoiled-farah}
For $n\geq 1$ let $I_n=[2^{n-1},2^n)$ and 
\[ \mrm{iFh}=\big\{F\cup E:F\in \mrm{Fh}\,\text{ and }\,E\subseteq I_n\,\text{ for some }\,n\big\}.\]
It is easy to check that $\mrm{iFh}$ satisfies ($U_{\neg c_0}$) but not ($U_S$). We show that $X_\mrm{iFh}$ contains a copy of $c_0$: Indeed, the normalized block basis sequence $x_n=2^{-2^n}\chi_{I_{2^n+1}}\in X_{\mrm{iFh}}$ is equivalent to the canonical basis of $c_0$ because 
\[ \bigg\|\sum_{n=1}^mx_n\bigg\|_{\mrm{iFh}}\leq\bigg\|\sum_{n=1}^mx_n\bigg\|_\mrm{Fh}+1\leq \sum_{n=1}^m2^{-2^n}\frac{2^{2^n+1-1}}{2^n+1}+1< 2.\]
\end{exa}

\begin{exa}[``$n^2$ Farah'' with intervals]\label{iFh-n2} 
We show an example of an $\mrm{iFh}(\mc{Q})$ family which does not satisfy ($U_{\neg c_0}$). Let $V_n=[3^{n-1},3^n)$ and 
\[ \mc{Q}_n=\bigg\{F\subseteq V_n:\big|F\cap V_n\big|\leq \frac{|V_n|}{n^2}\bigg\}.\]
Notice that $2\cdot 3^{n-1}/n^2\geq 1$ hence $\mc{Q}_n$ covers $V_n$. It is trivial to check that the sequence $(V_n)$ witnesses that $\mrm{iFh}(\mc{Q})$ does not satisfy ($U_{\neg c_0}$).
\end{exa}

\begin{exa}[``$\mc{S}_1$ in $\mc{S}_2$'']\label{exa-S1S2} We will construct an $\mrm{iFh}(\mc{Q})$ family which satisfies ($U_S$) but $X_{\mrm{iFh}(\mc{Q})}$ contains a copy of $c_0$. Let $(I_n)_{n\geq 1}$ be the decomposition of $\NN$ into maximal intervals from $\mc{S}_2$ (the $2$nd Schreier family). In other words, $I_n=\bigcup_{j=1}^{\min(I_n)} I^n_j$ where $I^n_j$ are maximal consecutive intervals from $\mc{S}_1$. Define
\[ \mc{Q}_n=\bigg\{E\subseteq I_n:E\,\text{ can be covered by }\,\leq\frac{2\min(I_n)}{n^2}\,\text{ many sets from }\,\mc{S}_1\bigg\}.\] 

We claim that $\mrm{iFh}(\mc{Q})$ is as desired. Why $s_n:=\lfloor 2\min(I_n)/n^2\rfloor$? Because we need that $s_n\geq 1$ to make sure $\mc{Q}_n$ covers $I_n$, and we will use below that $s_n\to \infty$ and $\sum_{n=1}^\infty \frac{s_n}{\min(I_n)}<\infty$. 

\smallskip
$\mrm{iFh}(\mc{Q})$ satisfies ($U_S$): Notice that if $E\notin\mc{S}_1$ is finite and $E=E_1\cup E_2\cup\dots \cup E_m$ is a decomposition into consecutive non-empty sets from $\mc{S}_1$, all of them maximal except perhaps the last one, then $|E_{m-1}\cup
	E_m|> |E|/2$. In general (by induction), assuming that $m>d$ we have \[ |E_{m-d}\cup E_{m-d+1}\cup\dots \cup E_m|>|E|(1-2^{-d}).\] 
Now let $(F_k)$ be a sequence of pairwise disjoint sets from $\mrm{iFh}(\mc{Q})\setminus\{\0\}$, we will find an $A\in\overline{\mrm{Fh}(\mc{Q})}\subseteq\overline{\mrm{iFh}(\mc{Q})}$ such that $\limsup_k|A\cap F_k|/|F_k|=1$. We can assume that 
\[ \big\{n:F_1\cap I_n\ne\0\}<\big\{n:F_2\cap I_n\ne\0\big\}<\dots.\] 
If $F_k\cap I_n\in\mc{Q}_n$ for every $n$, then we can add this $A_k=F_k$ to the desired $A$. Now assume that $F_k\cap I_{n_k}\notin\mc{Q}_{n_k}$ for a fixed $n_k$, $F_k\cap I_{n_k}=E_1\cup E_2\cup\dots E_m$ is a decomposition as above, it follows that $m>s_{n_k}$. Let 
\[ A'_k=E_{m-s_{n_k}+1}\cup E_{m-s_{n_k}+2}\cup\dots \cup E_m\in \mc{Q}_{n_k},\] 
and $A_k=A'_k\cup (F_k\setminus I_{n_k})$. Then $A_k\in\mrm{Fh}(\mc{Q})$ and  
\[ \frac{|A_k\cap F_k|}{|F_k|}\geq \frac{|A_k\cap F_k\cap I_{n_k}|}{|F_k\cap I_{n_k}|}=\frac{|A'_k|}{|F_k\cap I_{n_k}|}> 1-2^{-s_{n_k}+1}.\]
It follows that $A=\bigcup_kA_k\in\overline{\mrm{Fh}(\mc{Q})}$ and $\limsup_k |A\cap F_k|/|F_k|=1$.

\smallskip
$X_{\mrm{iFh}(\mc{Q})}$ contains a copy of $c_0$: We will need the second repeated average sequence $\lam^2_n\in [0,\infty)^\NN$, $n\geq 1$ (see e.g. \cite{Argyros98}) but instead of defining it precisely, let us only recall the properties we need. First of all, $\mrm{supp}(\lam^2_n)=I_n$ and $\|\lam^2_n\|_{\mrm{iFh}(\mc{Q})}=\sum_{k\in I_n}\lam^2_n(k)=1$, in particular, $(\lam^2_n)$ is a normalized block sequence in $X_{\mrm{iFh}(\mc{Q})}$. Furthermore, if $\lam^2$ is the formal sum of $(\lam^2_n)$, $G\in\mc{S}_1\setminus\{\0\}$, and $\min(G)\in I^n_j$, then 
\[ \sum_{k\in G}\lam^2(k)\leq \sum_{k\in I^n_j}\lam^2(k)=\sum_{k\in I^n_j}\lam^2_n(k)=\frac{1}{\min(I_n)}.\] 
It follows that  
\[ \bigg\|\sum_{n=1}^m\lam^2_n\bigg\|_{\mrm{iFh}(\mc{Q})}\leq 1+\sum_{n=1}^m\frac{2\min(I_n)}{n^2 \min(I_n)}< 5.\] 
\end{exa}

\subsection{The rapid Farah family}\label{rapid-farah} The following modification of $\mrm{Fh}$ opened Pandora's box: As above, let $I_n=[2^{n-1},2^n)$, given a set $A\subseteq\NN$ let $D_A=\{n\in\NN:A\cap I_n\ne\0\}$ and 
\[ s_A:\big\{1,2,\dots,|D_A|\big\}\to D_A\] 
be its increasing enumeration (for $A=\0$ let $s_A$ be the empty sequence), and define the \emph{rapid Farah family} as 
\[ \mrm{rFh}=\bigg\{F\in [\NN]^{<\infty}:\big|F\cap I_{s_F(n)}\big|\leq\frac{2^{s_F(n)-1}}{n}\,\text{ for every }\,n\in\dom(s_F)\bigg\}.\]
It is easy to see that $(I_n)$ witnesses that $\mrm{rFh}$ fails to satisfy ($U_S$), hence $X_{\mrm{rFh}}$ does not have the Schur property. We will show that the space $X_\mrm{rFh}$ is $\ell_1$-saturated. There are many other known examples of such spaces (see e.g. \cite{Hagler}, \cite{Popov}) but most of them are quite involved, certainly more complicated than $X_{\mrm{rFh}}$.

\smallskip
First of all, let us introduce some notations: For $1\leq l\leq 2^{n-1}$ we define the seminorm $\|\bullet\|_{n,l}$ on $\mbb{R}^\NN$ as 
\[ \|x\|_{n,l}=\max\bigg\{\sum_{i\in F}|x(i)|:F\subseteq I_n\,\text{ and }\,|F|\leq \frac{|I_n|}{l}\bigg\}.\]
Of course, $\|\bullet\|_{n,l}$ is a norm on $\{x\in\mbb{R}^\NN:\mrm{supp}(x)\subseteq I_n\}\simeq\mbb{R}^{I_n}$. For example, $\|x\|_\mrm{Fh}=\sum_{n=1}^\infty\|x\|_{n,n}$ and 
\[ \|x\|_\mrm{rFh}=\sup\bigg\{\sum_{k=1}^\infty\|x\|_{n_k,k}:(n_k)\in\NN^\NN\,\text{ is strictly increasing}\bigg\}.\] 

We prove an easy observation basically saying that the sequence $\|x\|_{n,l}$ does not decrease too fast in $l$ assuming $l\ll 2^n$: 
\begin{fact}
If $x\in\mbb{R}^\NN$, $1\leq l\leq l'$, and $(l'+1)^2\leq 2^{n-1}$, then
\[\tag{$\star$} \|x\|_{n,l'} \geq \dfrac{l}{l'+1} \|x\|_{n,l}.\]
\end{fact}
\begin{proof}
It is trivial to check that if $1\leq K'\leq K$, $v\in \mbb{R}^K$, and $v(1)\geq v(2)\geq\dots\geq v(K)\geq 0$, then $(v(1)+\dots+v(K'))/(v(1)+\dots+v(K))\geq K'/K$. It follows that
\begin{align*} 
\frac{\|x\|_{n,l'}}{\|x\|_{n,l}}\geq& \frac{\lfloor2^{n-1}/l'\rfloor}{\lfloor 2^{n-1}/l\rfloor} \geq \frac{2^{n-1}/l'-1}{2^{n-1}/l}= \frac{l}{l'} - \frac{l}{2^{n-1}} \\
\geq& \frac{l}{l'} - \frac{l}{(l'+1)^2}> \frac{l\cdot l'\cdot(l'+1)}{l'\cdot(l'+1)^2} = \frac{l}{l'+1}.\qedhere
\end{align*} 
\end{proof}

\begin{thm}\label{elel}
The space $X_{\mrm{rFh}}$ is $\ell_1$-saturated.
\end{thm}

\begin{proof} 
Applying the Selection Principle (now its other variant, see e.g. \cite[Theorem 4.26]{bible}), it is enough to find copies of $\ell_1$ in subspaces of the form $[(x_m)]$ where $(x_m)$ is a normalized block basic sequence. We can assume that the sets $D_m=\{n:\mrm{supp}(x_m)\cap I_n\ne\0\}$ are consecutive and fix 
\[ \big\{n^m_1<n^m_2<\dots<n^m_{l_m}\big\}\subseteq D_m\,\text{ such that }\,1=\|x_m\|_{\mrm{rFh}}=\sum_{k=1}^{l_m}\|x_m\|_{n^m_k,k}.\] 

The proof is based on the following technical statement: 
\begin{nnclaim}
Let $s\in \mathbb{N}$. Then there is a $y \in [(x_m)]$ such that the following holds:
\begin{itemize}\setlength\itemsep{0.1cm}
\item[(a)] $\mrm{supp}(y)\subseteq \NN\setminus \bigcup_{n=1}^sI_n$ is finite and $\|y\|_\mrm{rFh} = 1$. 
\item[(b)] If $z\in c_{00}$, $\mrm{supp}(z) \subseteq \bigcup_{n=1}^s I_n$, and $\beta\in\mbb{R}$, then 
\[\| z + \beta y\|_{\mrm{rFh}} \geq \| z
	\|_{\mrm{rFh}} +  |\beta|/2.\]
\end{itemize}
\end{nnclaim}
	
Let us first show that this implies the theorem. We can construct inductively a normalized block basic sequence $y_k\in [(x_m)]$ the following way: Let $y_1 = x_1$ and in general, let $y_{k+1}$ be $y$ from the Claim above to an $s$ satisfying $\mrm{supp}(y_k)\subseteq\bigcup_{n=1}^s I_n$. To finish the argument we show that $(y_k)$ is equivalent to the usual basis of $\ell_1$. If $K\in\NN$ and $\be\in\mbb{R}^K$ then 
\begin{align*}
\bigg\|\sum_{k=1}^K \be(k) y_k\bigg\|_\mrm{rFh} &\geq \bigg\|\sum_{k=1}^{K-1}\be(k) y_k\bigg\|_\mrm{rFh} + \frac{|\beta(K)|}{2} \\
&\geq \bigg\|\sum_{k=1}^{K-2}\be(k) y_k\bigg\|_\mrm{rFh} + \frac{|\be(K-1)|}{2} + \frac{|\beta(K)|}{2} \geq \dots\\ 
&\geq	|\beta(1)| + \frac{|\beta(2)|}{2} + \dots +  + \frac{|\beta(K)|}{2} \geq \frac{1}{2}\sum_{k=1}^K|\beta(k)|.
\end{align*} 

Regarding the Claim, we distinguish two cases.

\smallskip
\textbf{Case 1.} $\max\{\|x_m\|_{n^m_i,1}:i=1,\dots,l_m\}\xrightarrow{m\to\infty} 0$. 

We show that $y=x_m$ is as required if $m$ is large enough. Take an arbitrary $m\in\NN$ such that $s\leq \min(D_m)-4$. Then $s\leq n^m_1-4$ and hence $s+i+1\leq n^m_1+(i-1)-2\leq n^m_i-2$ for every $i\in [1,l_m]$. It follows that $(s+i+1)^2\leq 2^{n^m_i-1}$ for every such $i$. The point is that, assuming $s\leq \min(D_m)-4$ and $1\leq i\leq l_m$,
\begin{itemize}\setlength\itemsep{0.1cm}
\item[(i)] $\|x_m\|_{n^m_i,s+i}$ is defined, and 
\item[(ii)] ($\star$) applies with $x=x_m$, $l=i$, $l'=s+i$, and $n=n^m_i$. 
\end{itemize}

By the definition of $\|\bullet\|_\mrm{rFh}$, we know that
\[\| z + \beta x_m \|_{\mrm{rFh}} \geq \| z \|_{\mrm{rFh}} + |\beta| \sum_{i=1}^{l_m}\|x_m\|_{n^m_i,s + i}.\]
Therefore, given any $r\in [1,l_m)$,
\begin{align*}\tag{1} \| z + \beta x_m \|_{\mrm{rFh}} - \| z \|_{\mrm{rFh}} - \| \beta x_m \|_{\mrm{rFh}} 
&\geq |\beta|\bigg(\sum_{i=1}^{l_m} \|x_m\|_{n^m_i,s+i}- \sum_{i=1}^{l_m} \|x_m\|_{n^m_i,i}\bigg)\\ 
&\geq |\beta|\bigg(\sum_{i=r+1}^{l_m} \|x_m\|_{n^m_i,s+i}- \sum_{i=1}^{l_m} \|x_m\|_{n^m_i,i}\bigg).\end{align*}

Now, we need to specify $m$ a little further: Fix first $r$ then $m$ from $\NN$ such that
\begin{itemize}\setlength\itemsep{0.1cm}
\item[($r$)] $r/(s + r + 1) \geq 3/4$;
\item[($m$)] $s\leq \min(D_m)-4$ and $\|x_m\|_{n^m_i,1}\leq 1/4r$ for every $i\in [1,l_m]$.
\end{itemize}
Applying ($\star$) as in (ii) above, for every $i\in (r,l_m]$ we have  
\[ \|x_m\|_{n^m_i,s + i} \geq \frac{i}{s+i+1}\|x_m\|_{n^m_i,i} \geq  \frac{r}{s+r+1}\|x_m\|_{n^m_i,i} \geq \frac{3}{4}\|x_m\|_{n^m_i,i},\] 
and hence the last difference of sums in (1) can be estimated as follows:
\begin{align*}\tag{2} 
\sum_{i=r+1}^{l_m} \|x_m\|_{n^m_i,s+i}- \sum_{i=1}^{l_m} \|x_m\|_{n^m_i,i}&\geq -\frac{1}{4}\sum_{i=r+1}^{l_m} \|x_m\|_{n^m_i,i} - \sum_{i=1}^{r} \|x_m\|_{n^m_i,i}\\
&\geq -\frac{1}{4}\|x_m\|_\mrm{rFh} - r \|x_m\|_{n^m_i,1}\geq -\frac{1}{4} - r\frac{1}{4r} = - \frac{1}{2}.
\end{align*}

Combining (1) and (2), $\| z + \beta x_m \|_{\mrm{rFh}} - \| z \|_{\mrm{rFh}} - |\beta| \geq -|\beta|/2$, hence $y=x_m$ is as desired.

\smallskip
\textbf{Case 2.} There are a $\de>0$, an $S\in[\NN]^\infty$, and for every $m\in S$ an $i_m\in [1,l_m]$ such that $\|x_m\|_{n^m_{i_m},1}\geq \de$. 

Fix $J\in\NN$ and $E=\{m_1<m_2<\dots<m_J\}\subseteq S\setminus \{1,2,3\}$. Then, with $n_j=n^{m_j}_{i_{m_j}}$, we know that $1\leq m_1-3\leq n^{m_1}_1-3\leq n_1-3$, it follows that $j+1\leq n_1+(j-1)-2\leq n_j-2$, and hence $(j+1)^2\leq 2^{n_j-1}$ and we can apply ($\star$) with $l=1$, $l'=j$, and $n=n_j$:
\[ \bigg\|\sum_{m\in E} x_m\bigg\|_\mrm{rFh} \geq \sum_{j=1}^J\|x_{m_j}\|_{n_j,j} \geq \sum_{j=1}^J\frac{\|x_{m_j}\|_{n_j,1}}{j+1}\geq \sum_{j=1}^J\frac{\de}{j+1}. \]
Therefore, we can pick finite subsets $E_1<E_2<\dots$ of $S$ such that $\|\sum_{m\in E_k}x_m\|_\mrm{rFh}\geq k$ for every $k$ and define \[ \wt{x}_k = \frac{\sum_{m\in E_k}x_m}{\|\sum_{m\in E_k}x_m\|_\mrm{rFh}}\in [(x_m)],\]
a normalized block basic sequence. Instead of working with $(x_m)$, we switch to $(\wt{x}_k)$ and define everything as above, $\wt{D}_k=\{n:\mrm{supp}(\wt{x}_k)\cap I_n\ne\0\}$, $\{\wt{n}^k_i:i=1,\dots,\wt{l}_k\}\subseteq\wt{D}_k$ such that $1=\|\wt{x}_k\|_\mrm{rFh}=\sum_{i=1}^{\wt{l}_k}\|\wt{x}_k\|_{\wt{n}^k_i,i}$, etc. Then 
\[ \max\big\{\|\wt{x}_k\|_{\wt{n}^k_i,1}:i\in[1,\wt{l}_k]\big\}\leq \frac{\max\big\{\|x_m\|_{\wt{n}^k_i,1}:m\in E_k,i\in [1,\wt{l}_k]\big\}}{k}\leq \frac{1}{k},\] therefore, we can apply Case 1 to find the desired $y\in [(\wt{x}_k)]\subseteq [(x_m)]$.
\end{proof}

\begin{rem} In fact, we obtained an even simpler example of an $\ell_1$-saturated space without the Schur property: Consider $X=[(x_n)] \subseteq X_\mrm{rFh}$ where $x_n = \chi_{I_n}/|I_n|$, then $(x_n)$ witnesses the failure of the Schur property, and, by the last theorem, $X$ is $\ell_1$-saturated. Considering $X\subseteq\mbb{R}^\NN$ along the $1$-unconditional basis $(x_n)$, the norm is of the following very simple form:
\[ \|a\|= \sup\bigg\{\sum_{k=1}^\infty\frac{|a(n_k)|}{k}:(n_k)\in \NN^\NN\,\text{ is strictly increasing}\bigg\}. \]
	In other words, $X$ is the completion of $c_{00}$ w.r.t. $\|\bullet\|$; or, alternatively,  $\|\bullet\|$ is an extended norm on $\mbb{R}^\NN$ and $a\in X$ iff $\|a\|<\infty$ iff $\|P_{[n,\infty)}(a)\|\to 0$, because $X_{\mrm{rFh}}$ does not
	contain copies of $c_0$, hence nor does $X$, therefore its basis is boundedly complete.
\end{rem}

\subsection{Combinatorial spaces with prerequisite subspaces} Example \ref{exa-antichains} and Example \ref{exa-chains} witness that combinatorial spaces can contain any Banach space with an unconditional basis. However, it is more or less impossible to track
down and ``really'' see e.g. a copy of $\ell_2$ in these examples. We will show that for every Banach space $Y$ with unconditional basis, there is a natural family $\mc{F}$ such that a complemented block basic sequence in $X_\mc{F}$ is equivalent to the basis of $Y$. The point is that we may encode a given ``geometric'' structure in the definition of $\mc{F}$. 

Let $I_n=[2^{n-1},2^n)$, $\Omega=\bigcup_{n=3}^\infty I_n$, and fix a Banach space $Y$ with normalized $1$-unconditional basis $(b_n)_{n\geq 3}$. We consider $Y\subseteq\mbb{R}^{\NN\setminus\{1,2\}}$ along this basis (that is, $y=\sum_{n=3}^\infty y(n)b_n$), also, we consider $Y^*\subseteq\mbb{R}^{\NN\setminus\{1,2\}}$ along $(b^*_n)$. As $(b_n)$ is $1$-unconditional, if $\sigma\in Y^*$ then $\|\sigma\|_{Y^*}\leq \sum_{n=3}^\infty|\sigma(n)|$. Define
\[ \mc{F}(Y)=\bigg\{F\in [\Omega]^{<\infty}:\bigg(\frac{|F\cap I_n|}{|I_n|}\bigg)\in B(Y^*)\bigg\}\]
and notice that it is a hereditary cover of $\Omega$.

\begin{thm}\label{prere}
With $Y$ and $\mathcal{F}=\mc{F}(Y)$ as above, the sequence $x_n = \chi_{I_n}/|I_n|$ is a complemented normalized block basic sequence in $X_\mc{F}$ that is equivalent to $(e_n)$.
\end{thm}
\begin{proof}
If $y\in c_{00}(\Omega)$, then
\begin{align*} \bigg\| \sum_{n=3}^\infty y(n)x_n \bigg\|_\mc{F}&= \sup\bigg\{\sum_{n=3}^\infty |F\cap I_n|\frac{|y(n)|}{|I_n|}:F\in\mc{F}\bigg\}\\
&= \sup\bigg\{\bigg|\bigg\la\bigg(\eps_n\frac{|F\cap I_n|}{|I_n|}\bigg),y\bigg\ra\bigg|:\eps_n=\pm 1\,\text{ and }\,F\in \mc{F}\bigg\}\\
&\leq\sup\big\{|\la\sigma,y\ra|:\sigma\in B(Y^*)\big\}=\|y\|_Y.
\end{align*}

Conversely, given $\sigma\in B(Y^*)$, for each $n\geq 3$ we can fix an $F_n\subseteq I_n$ such that 
\[ \frac{|F_n|}{|I_n|}\leq |\sigma(n)|<\frac{|F_n|+1}{|I_n|}.\]
Then $A_\sigma=\bigcup_{n=3}^\infty F_n\in \overline{\mc{F}}$ and
\[ \bigg\|\sigma-\bigg(\mrm{sgn}(\sigma(n))\frac{|A_\sigma\cap I_n|}{|I_n|}\bigg)\bigg\|_{Y^*} \leq \sum_{n=3}^\infty\bigg|\sigma(n)-\mrm{sgn}(\sigma(n))\frac{|F_n|}{|I_n|}\bigg|<\sum_{n=3}^\infty \frac{1}{|I_n|}=\frac{1}{2}.\]
Therefore, if $y\in c_{00}(\Omega)$, then 
\begin{align*} 
\|y\|_Y&=\sup\big\{|\la\sigma,y\ra|:\sigma\in B(Y^*)\big\}\\
&\leq \sup\bigg\{\bigg|\bigg\la\bigg(\eps_n\frac{|F\cap I_n|}{|I_n|}\bigg),y\bigg\ra\bigg|+\frac{\|y\|_Y}{2}:\eps_n=\pm 1\,\text{ and }\,F\in \mc{F}\bigg\}\\
	&=\bigg\|\sum_{n=3}^\infty y(n)x_n\bigg\|_\mc{F}+\frac{\|y\|_Y}{2},
\end{align*}  
and hence $\|y\|_Y\leq 2\|\sum_{n=3}^\infty y(n)x_n\|_\mc{F}$.

\smallskip
To show that $[(x_n)]$ is complemented in $X_\mc{F}$, define $T:\mbb{R}^\NN \rightarrow \mbb{R}^\NN$ as follows: For $x\in \mbb{R}^\NN$ and $k \in I_n$ let \[ T(x)(k) = \sum_{i \in I_n} \frac{x(i)}{2^{n-1}}.\] 
In other words, $T(x)$ on $I_n$ replaces the values of $x$ with its arithmetic mean over $I_n$. Clearly, $T$ is linear, $T\clrest [(x_n)]$ is the identity, and $T^2=T$. It remains to show that $T[X_\mc{F}]\subseteq X_\mc{F}$ (i.e. $T[X_\mc{F}]\subseteq [(x_n)]$) and that $T$ is continuous.

Given $x \in X_{\mathcal{F}}$ and $F \in \mathcal{F}$, let $E\subseteq\NN$ be such that 
\begin{itemize}\setlength\itemsep{0.1cm}
\item[(a)] $|E\cap I_n|= |F\cap I_n|$ for every $n$ (hence $E\in\mc{F}$), and 
\item[(b)] $\sum_{k\in E}|x(k)|$ is maximal with respect to (a). 
\end{itemize} 
It follows that $\sum_{k\in F}|T(x)(k)|\leq \sum_{k\in E}|x(k)|\leq \|x\|_\mc{F}$ holds for every $F\in\mc{F}$, hence $\|T(x)\|_\mc{F}\leq \|x\|_\mc{F}$. Applying this inequality, if $x\in X_\mc{F}$ and $n\geq 2$ then
\[ \|P_{[2^n,\infty)}(T(x))\|_\mc{F}=\|T(P_{[2^n,\infty)}(x))\|_\mc{F}\leq \|P_{[2^n,\infty)}(x)\|_\mc{F},\] 
therefore, $T(x)\in X_\mc{F}$, and so $T:X_\mc{F}\to X_\mc{F}$ is bounded.
\end{proof}

\begin{exa}\label{exa-l2}
Let $Y=\ell_2$ and let $\mc{F}$ be the associated family above. We show that $X_{\mc{F}}$ does not contain $c_0$ and hence this example witnesses that not containing $c_0$ does not imply $\ell_1$-saturation in the realm of combinatorial spaces. Also, notice that $\mc{F}$ fails to satisfy ($U_S$) because of the sequence $(I_n)$. 

Let $(x_n)$ be an $\mc{F}$-supported normalized block basic sequence in $X_\mc{F}$, $\mrm{supp}(x_n)=F_n\in\mc{F}$; by thinning our sequence, we can assume that the sets $D_n=\{k\geq 3:F_n\cap I_k\ne\0\}$ are consecutive and 
\[\tag{$\ast$} \sum_{n=1}^\infty\frac{16}{2^{\min(D_n)}}<\frac{1}{4}.\]

For $k\in D_n$, let $F_{n,k}=F_n\cap I_k$ and pick an $E_{n,k}\subseteq F_{n,k}$ such that 
\[\tag{$\ast\ast$} |E_{n,k}|=\bigg\lceil\frac{|F_{n,k}|}{2n}\bigg\rceil\,\text{ and }\,\|P_{E_{n,k}}(x_n)\|_\mc{F}\geq\frac{\|P_{F_{n,k}}(x_n)\|_\mc{F}}{2n}.\] 
	We show that \[ A=\bigcup_{n=1}^\infty\bigcup_{k\in D_n}E_{n,k}\in\overline{\mc{F}}\] and that $\sum_{n=1}^\infty\|P_A(x_n)\|_\mc{F}=\infty$ (hence ($\neg c_0$) holds).

\smallskip
$A\in\overline{\mc{F}}$: 
\begin{align*} 
\sum_{k=3}^\infty\frac{|A\cap I_k|^2}{|I_k|^2}&=\sum_{n=1}^\infty\sum_{k\in D_n}\frac{|E_{n,k}|^2}{|I_k|^2}\leq \sum_{n=1}^\infty\sum_{k\in D_n}\bigg(\frac{|F_{n,k}|}{2n}+1\bigg)^2\frac{1}{|I_k|^2} \\
&= \sum_{n=1}^\infty\bigg(\frac{1}{(2n)^2}\sum_{k\in D_n}\frac{|F_{n,k}|^2}{|I_k|^2} +\frac{1}{n}\sum_{k\in D_n}\frac{|F_{n,k}|}{|I_k|^2}+\sum_{k\in D_n}\frac{1}{|I_k|^2}\bigg)
\end{align*}
where we know the following:
\[\tag{1} \sum_{k\in D_n}\frac{|F_{n,k}|^2}{|I_k|^2}=\sum_{k=3}^\infty\frac{|F_n\cap I_k|^2}{|I_k|^2}\leq 1\,\text{ because }\,F_n\in\mc{F}.\]
\[\tag{2} \frac{1}{n}\sum_{k\in D_n}\frac{|F_{n,k}|}{|I_k|^2}\leq  \frac{1}{n}\sum_{k\in D_n}\frac{1}{2^{k-1}}\leq \frac{1}{n}\sum_{k=\min(D_n)}^\infty\frac{1}{2^{k-1}}=\frac{4}{n\cdot 2^{\min(D_n)}}< \frac{16}{2^{\min(D_n)}}. \]
	\[\tag{3} \sum_{k\in D_n}\frac{1}{|I_k|^2}\leq \sum_{k=\min(D_n)}^\infty\frac{1}{2^{2k-2}}=\frac{16}{3\cdot 2^{2\min(D_n)}}<\frac{16}{2^{\min(D_n)}}.\]
Now, substituting (1), (2), and (3) in the estimation above and applying ($\ast$):
\[
\sum_{k=3}^\infty\frac{|A\cap I_k|^2}{|I_k|^2}< \sum_{n=1}^\infty\bigg(\frac{1}{(2n)^2}+\frac{16}{2^{\min(D_n)}}+\frac{16}{2^{\min(D_n)}}\bigg)<\frac{\pi^2}{24}+\frac{1}{4}+\frac{1}{4}<1.\]

The second statement follows easily from ($\ast\ast$):
\begin{align*}
\sum_{n=1}^\infty\|P_A(x_n)\|_\mc{F}&=\sum_{n=1}^\infty\sum_{k\in D_n}\|P_{E_{n,k}}(x_n)\|_\mc{F}\geq\sum_{n=1}^\infty\sum_{k\in D_n}\frac{\|P_{F_{n,k}}(x_n)\|_\mc{F}}{2n}\\
&=\sum_{n=1}^\infty\frac{\|P_{F_n}(x_n)\|_\mc{F}}{2n}=\sum_{n=1}^\infty\frac{1}{2n}=\infty.
\end{align*}
\end{exa}

\subsection{Universal families and spaces.} We will finish our journey to the Zoo with a Fra\"ise type construction which will provide us with another classical example of a separable Banach space, the Pe\l czy\'nski's universal space, of the form $X_\mc{F}$.

\smallskip
We say that an $\mc{F}\in\mrm{FHC}$ is \emph{universal} if every $\mc{G}\in\mrm{FHC}$ is isomorphic to a restriction of $\mc{F}$, that is, there is a one-to-one $e:\NN\to\NN$ such that $G\in\mc{G}$ iff $e[G]\in\mc{F}$. Clearly, if $\mc{F}$ is universal, then $X_\mc{F}$ contains complemented copies of all combinatorial spaces (in other words, $X_\mc{F}$ is universal for this class of spaces). To show that there are universal families, we introduce the following notion: 

\begin{df} 
An $\mc{F}\in\mrm{FHC}$ satisfies the \emph{extension property}, EP if the following holds: If $\mc{H}$ is a hereditary cover of a finite set $E$, $E_0\subseteq E$, and $\al_0:E_0\to \NN$ is an isomorphism between $\mc{H}\clrest E_0$ and $\mc{F}\clrest \al_0[E_0]$, then there is an isomorphism $\al:E\to \NN$ between $\mc{H}$ and $\mc{F}\clrest\al[E]$ extending $\al_0$.
\end{df}

\begin{prop}\label{fraise} The following holds:
\begin{itemize}\setlength\itemsep{0.1cm}
\item[(a)] There is a $\mc{P}$ satisfying EP.
\item[(b)] If $\mc{P}$ satisfies EP then it is universal.
\item[(c)] If $\mc{P}$ and $\mc{G}$ satisfy EP, then they are isomorphic.
\end{itemize}
\end{prop}
\begin{proof} (a): Notice that we can always assume that $|E|=|E_0|+1$ and that $\al_0=\mrm{id}_{E_0}$. Now, $\mc{P}$ will be constructed by recursion of the form $\bigcup_{N\in S}\mc{P}_N$ where $\mc{P}_N$ is a hereditary cover of $\{1,2,\dots,N\}$, $S\subseteq\NN$ is infinite, and $\mc{P}_N\clrest \{1,2,\dots,M\}=\mc{P}_M$ whenever $M<N$, $M,N\in S$. Assume that we already have $\mc{P}_N$ for some $N$ (let $\mc{P}_0=\{\0\}$) and let $\{\mc{H}_i:i=1,2,\dots,K\}$ be an enumerations of all hereditary covers of $\{1,2,\dots,N,N+1\}$ satisfying $\mc{H}_i\clrest \{1,2,\dots,N\}=\mc{P}_N$. We can relabel $N+1$ to $N+i$ in the underlying set of $\mc{H}_i$ and define $\mc{P}_{N+K}=\bigcup_{i=1}^K\mc{H}_i$.

\smallskip
(b): Fix a $\mc{G}\in\mrm{FHC}$ and assume that we already defined the restriction $e_k$ of the desired embedding $e$ on $\{1,2,\dots,k\}$, that is, $e_k$ is an isomorphism between $\mc{G}\clrest \{1,\dots,k\}$ and $\mc{P}\clrest e_k[\{1,\dots,k\}]$. To obtain $e_{k+1}$ we simply apply the extension property for $\mc{H}=\mc{G}\clrest\{1,\dots,k,k+1\}$, $E_0=\{1,\dots,k\}$, and $\al_0=e_k$. Then $e=\bigcup_{k=1}^\infty e_k$ is an embedding of $\mc{G}$ in $\mc{F}$.

\smallskip
(c) follows from a ``zigzagging'' argument: We define an isomorphism $\al$ between $\mc{P}$ and $\mc{G}$ by recursion such that at odd stages ($2n-1$), applying that $\mc{G}$ satisfies EP, we make sure that $n\in\mrm{dom}(\al)$, and at even stage ($2n$), applying that $\mc{P}$ satisfies EP, we make sure that $n\in\mrm{ran}(\al)$.  
\end{proof}

The following result provides us with a very simple presentation of Pe\l czy\'nski's universal space, namely as a combinatorial space. 

\begin{cor} Let $\mc{P}$ be an universal family. Then the space $X_\mathcal{P}$ contains complemented copies of all Banach spaces with an unconditional basis. Consequently, $X_\mathcal{P}$ is isomorphic to Pe\l czy\'nski's universal space.
\end{cor}
\begin{proof} 
Let $Y$ be a Banach space with an unconditional basis. Theorem \ref{prere} gives us an $\mathcal{F}$ such that $X_\mathcal{F}$ contains a complemented copy of $Y$. The space $X_\mc{P}$ contains a complemented copy of $X_\mc{F}$ and hence of $Y$ as well. By \cite[Corollary 4]{universalPelczynski}, $X_\mathcal{P}$ is isomorphic to Pe\l czy\'nski's universal space.
\end{proof}

\begin{rem}
The family $\mc{P}$ satisfying EP is the Fra\"isse limit of the family of all finite hereditary families. In \cite{Garbulinska} the authors constructed Pe\l
czy\'nski's space as the Fra\"isse limit of a certain family of finite dimensional Banach spaces.
\end{rem}

\begin{rem}
	The above example provides a solution for one of Pe\l czy\'nski's questions, \cite[Problem 4]{universalPelczynski}, which seems to be still open. The canonical basis $(e_n)$ of $X_\mc{P}$, where $\mc P$ is a universal family, is not permutatively equivalent
	to Pe\l czy\'nski's universal unconditional basis $(u_n)$ of his universal space (see \cite[Problem 4]{universalPelczynski}), i.e. there is no permutation $\pi$ such that $(e_{\pi(n)})$ is equivalent to $(u_n)$. Indeed, contrary to the case of
	$(u_n)$, the base of our space is not universal. E.g. no subsequence of $(e_n)$ is equivalent to the canonical basis of $\ell_2$ because, assuming $H\subseteq\NN$ is infinite, either $X_{\mc{P}\upharpoonright H}=[(e_n)_{n\in H}]$ contains a copy of $\ell_1$ or $X_{\mc{P}\upharpoonright H}$ is $c_0$-saturated. 
\end{rem}

\begin{rem} Assuming $\mc{P}$ is universal, the space $X_\mathcal{P}$ contains a complemented copy of $B=X_\mc{C}$ (see Example \ref{exa-chains}) which, on the other hand, contains copies of every Banach space with unconditional basis. But, of course, from this fact we cannot a priori conclude that these copies are complemented. 
\end{rem}

\section{The ideal $\mc{H}(\mc{F},\neg c_0)$ and its relatives}\label{J ideals}

As the space $X_\mc{F}$ can be seen as an amalgamation of $c_0$ and $\ell_1$, it is natural to somehow measure how far $X_\mc{F}$ is from e.g. $c_0$:
\[ \mc{H}(\mc{F},\neg c_0)=\big\{H\subseteq\NN:X_{\mc{F}\upharpoonright H}\,\text{ does not contain copies of }\,c_0\big\}\]
where $\mc{F}\clrest H=\{F\in\mc{F}:F\subseteq H\}$ and hence $X_{\mc{F}\upharpoonright H}=[(e_n)_{n\in H}]$ is the closed linear span of $\{e_n:n\in H\}$ in $X_\mc{F}$. Clearly, this family is hereditary and contains all finite subsets of $\NN$. For example, $\mc{F}$ is compact iff $\mc{H}(\mc{F},\neg c_0)=[\NN]^{<\infty}$. Similarly, we can define $\mc{H}(\mc{F},\neg \ell_1)$, and in general, given a (hereditary) property $\Phi$ of Banach spaces, we can define $\mc{H}(\mc{F},\Phi)$.

\smallskip
Bringing $\mrm{FIN}(\mc{F})$ into play as well, recall that $X_\mc{F}=\mrm{EXH}(\mc{F})$ does not contain $c_0$ iff $(e_n)$ is boundedly complete iff $\mrm{EXH}(\mc{F})=\mrm{FIN}(\mc{F})$. In particular, 
\[ \mc{H}(\mc{F},\neg c_0)=\big\{H\subseteq\NN:\mrm{EXH}(\mc{F}\clrest H)=\mrm{FIN}(\mc{F}\clrest H)\big\}.\] It follows easily that $\mc{H}(\mc{F},\neg c_0)$ is closed under taking unions of finitely many elements from it, hence it is an ideal. Also, $\mc{H}(\mc{F},\neg c_0)$ is the coprojection to the first coordinate of the $G_{\delta\sigma}$ set 
\[ \big\{(H,x)\in [\NN]^\infty\times\mbb{R}^\NN:x\clrest H\in\mrm{FIN}(\mc{F}\clrest H)\setminus\mrm{EXH}(\mc{F}\clrest H)\big\},\] 
therefore $\mc{H}(\mc{F},\neg c_0)$ is coanalytic.

\begin{prop}\label{thm:no c0}
Let $\mc{F}\in\mrm{FHC}$. Then the following are equivalent:
\begin{itemize}\setlength\itemsep{0.1cm}
\item[(i)] $X_\mc{F}$ does not contain $c_0$.
\item[(ii)] $\mc{H}(\mc{F},\neg c_0)=\mc{P}(\NN)$.
\item[(iii)] $\mc{H}(\mc{F},\neg c_0)$ is not null in $\mc{P}(\NN)\simeq 2^\NN$.
\item[(iv)] $\mc{H}(\mc{F},\neg c_0)$ is not meager in $\mc{P}(\NN)\simeq 2^\NN$. 
\end{itemize} 
\end{prop}
\begin{proof} 
\smallskip
(i)$\leftrightarrow$(ii), (ii)$\to$(iii), and (ii)$\to$(iv) are trivial. 

\smallskip
(iii)$\to$(ii) and (iv)$\to$(ii) can be shown as follows: We say that an $\mc{H}\subseteq\mc{P}(\NN)$ is a \emph{tail-set} if $\mc{H}$ is closed for finite modifications (that is, if $H\in\mc{H}$, $A\subseteq\NN$, and $|H\triangle A|<\infty$, then $A\in\mc{H}$ as well). For example, ideals on $\NN$ are tail-sets. We know (see e.g. \cite[Theorem 21.3 and 21.4]{Oxtoby}) that if $\mc{H}\subseteq\mc{P}(\NN)$ is a measurable tail-set, then $\mc{H}$ is of measure $0$ or $1$; and if $\mc{H}$ is a tail-set with the Baire property (BP), then $\mc{H}$ is meager or comeager. Also, the measure preserving homeomorphism $C:\mc{P}(\NN)\to\mc{P}(\NN)$, $H\mapsto \NN\setminus H$ witnesses that measurable ideals are of measure $0$ or $=\mc{P}(\NN)$ and that ideals with the BP are meager or $=\mc{P}(\NN)$ (because if $\mc{H}\ne\mc{P}(\NN)$ then $\mc{H}\cap C[\mc{H}]=\0$). We just have to apply these results to the ideal $\mc{H}(\mc{F},\neg c_0)$ (and recall that coanalytic sets are measurable and have the BP).
\end{proof}

Regarding $\mc{H}(\mc{F},\neg \ell_1)$, applying Theorem \ref{noell1}, it follows that
\begin{align*} 
\mc{H}(\mc{F},\neg \ell_1)&=\big\{H\subseteq\NN:X_{\mc{F}\upharpoonright H}\,\text{ does not contain }\,\ell_1\big\}\\
&=\big\{H\subseteq\NN:\mc{F}\clrest H\,\text{ is compact in }\,\mc{P}(H)\big\}\\
&=\mc{P}(\NN)\setminus\big\{A\subseteq\NN:\exists\,B\in\overline{\mc{F}}\cap [\NN]^\infty\;A\supseteq B\big\}
\end{align*}
is also an ideal; and just like its cousin, $\mc{H}(\mc{F},\neg\ell_1)$ is also $\mbf{\Pi}^1_1$ because it is the coprojection to the first coordinate of the closed set 
\[ \big\{(A,B)\in\mc{P}(\NN)\times \big([\NN]^\infty\cap\overline{\mc{F}}\big):A\supseteq B\big\}.\]
In particular, the $\ell_1$ versions of Proposition \ref{thm:no c0} holds, and it says that the following are equivalent: (i) $X_\mc{F}$ does not contain $\ell_1$. (ii) $\mc{H}(\mc{F},\neg\ell_1)=\mc{P}(\NN)$. (iii) $\mc{H}(\mc{F},\neg\ell_1)$ is not null. (iv) $\mc{H}(\mc{F},\neg\ell_1)$ is not meager. 

\smallskip
Let us see some concrete examples. As above, let $I_n=[2^{n-1},2^n)$, $n\geq 1$, and for $H\subseteq \NN$ let us denote $D_H=\{n:H\cap I_n\ne\0\}$. Also, let us recall the following definitions:
\begin{align*}
\mrm{Fh}&=\bigg\{F\in [\NN]^{<\infty}:|F\cap I_n|\leq \frac{2^{n-1}}{n}\,\text{ for every }\,n\bigg\}\\
\mrm{iFh}&=\big\{F\cup E:F\in\mrm{Fh}\,\text{ and }\,E\subseteq I_n\,\text{ for some }\,n\big\}
\end{align*}

Clearly, $\mc{H}(\mrm{iFh},\neg\ell_1)=[\NN]^{<\infty}$.

\begin{prop}\label{Fhtilde} $\mc{H}(\mrm{iFh},\neg c_0)=\{H\subseteq\NN:\inf_{n\in D_H}\frac{2^{n-1}}{n|H\cap I_n|}>0\}$ and hence it is
$F_\sigma$.
\end{prop}
\begin{proof}
The family on the right side is clearly $F_\sigma$ ($\inf\0=\infty$ by definition).

\smallskip
First let $H\in [\NN]^\infty$ such that $\frac{2^{n-1}}{n|H\cap I_n|}\xrightarrow{n\in D} 0$ for some $D\in [D_H]^\infty$, w.l.o.g. we can assume that 
\[ s:=\sum_{n\in D} \frac{2^{n-1}}{n|H\cap I_n|}<\infty.\] For $n\in D$ define $x_n=\chi_{H\cap I_n}/|H\cap I_n|$. Then $(x_n)_{n\in D}$ is a normalized block basic sequence in $X_{\mrm{iFh}\upharpoonright H}$ and if $D=\{n_1<n_2<\dots\}$ then
\begin{align*}
\bigg\|\sum_{i=1}^mx_{n_i}\bigg\|_{\mrm{iFh}}&\leq\bigg\|\sum_{i=1}^mx_{n_i}\bigg\|_\mrm{Fh}+1\leq \sum_{i=1}^m\frac{\min\big(|H\cap I_{n_i}|,\frac{2^{n_i-1}}{n_i}\big)}{|H\cap I_{n_i}|}+1\\
&\leq \sum_{i=1}^m\frac{2^{n_i-1}}{n_i|H\cap I_{n_i}|}+1\leq s+1,
\end{align*} 
and hence, $(x_n)_{n\in D}$ is equivalent to the canonical basis of $c_0$.

\smallskip
Conversely, fix an $H\in [\NN]^\infty$ from the family on the right and an $\eps>0$ such that $\lfloor 2^{n-1}/n\rfloor\geq \eps|H\cap I_n|$ for every $n\in D_H$. We show that no normalized block basic sequence $(x_i)$ in
	$X_{\mrm{iFh}\upharpoonright H}$ can be equivalent to the usual basis of $c_0$. By thinning out such a sequence, we can assume that the sets $D_i=\{n:\mrm{supp}(x_i)\cap I_n\ne\0\}\subseteq D_H$ are consecutive. For each $i$ and $n\in D_i$ we can pick an $H_{i,n}\subseteq\mrm{supp}(x_i)\cap I_n\subseteq H\cap I_n$ of size $\leq 2^{n-1}/n$ such that 
\[ \sum_{k\in H_{i,n}}|x_i(k)|\geq \eps \sum_{k\in H\cap I_n}|x_i(k)|,\] 
and let $H_i=\bigcup_{n\in D_i}H_{i,n}\in\mrm{Fh}$. It follows that $\|P_{H_i}(x_i)\|_{\mrm{Fh}}\geq \eps \|x_i\|_1\geq\eps\|x_i\|_\mrm{iFh}=\eps$, and, of course, $\bigcup_{i=1}^nH_i\in\mrm{Fh}$ for every $n$, hence if $a\in\mbb{R}^n$, then  
\[\bigg\|\sum_{i=1}^n a(i)x_i\bigg\|_{\mrm{iFh}}\geq\bigg\|\sum_{i=1}^n a(i)x_i\bigg\|_{\mrm{Fh}} \geq \sum_{i=1}^n|a(i)| \|P_{H_i}(x_i)\|_{\mrm{Fh}}\geq \eps \sum_{i=1}^n|a(i)|,\] 
therefore, $(x_i)$ (after thinning out at the beginning) is equivalent to the canonical basis of $\ell_1$.  
\end{proof}

We recall the definition of the family $\mc{A}$ from Example \ref{exa-antichains}:
\[ \mc{A}=\big\{F\subseteq 2^{<\NN}:F\,\text{ is a finite antichain}\big\}.\]

\begin{prop}\label{A-no-c0} $\mc{H}(\mc{A},\neg c_0)=\big\{H\subseteq 2^{<\NN}:H$ does not contain infinite chains$\}$ and it is $\mbf{\Pi}^1_1$-complete in $\mc{P}(2^{<\NN})$.
\end{prop}
\begin{proof} To prove the non-trivial direction, we show that if $H\subseteq 2^{<\NN}$ is infinite and does not contain infinite chains, then $X_{\mc{A}\upharpoonright H}$ does not contain $c_0$. We can assume that the empty sequence belongs to $H$, consider the well-founded tree $(H,\subseteq)$, and define the usual rank function $\rho:H\to\mrm{On}$ on this tree, that is, $\rho$ is $0$ on the terminal nodes, and if we already know $H_{\rho<\al}=\{t\in H:\rho(t)<\al\}$, then define $\rho(s)=\al$ if $s\notin H_{\rho<\al}$ and all immediate successors of $s$ (in $H$) belong to $H_{\rho<\al}$.  This recursion gives us a full function on $H$.

We prove by induction on $\rho(\0)$ that $\mrm{EXH}(\mc{A}\clrest H)=\mrm{FIN}(\mc{A}\clrest H)$. If $\rho(\0)=1$ then the first ($=$last) level $L$ of $H$ is an infinite antichain, hence $\mc{A}\clrest H=[L]^{<\infty}\cup\{\{\0\}\}$, and so $\mrm{EXH}(\mc{A}\clrest H)=\mrm{FIN}(\mc{A}\clrest H)$ is isomorphic to $\ell_1$. At stage $\al=\rho(\0)$, let $\{s_i:i=1,2,\dots\}$ (a finite or infinite set) be an enumeration of the first level of $H$, and let $H_i=\{t\in H:t$ extends $s_i\}$. As $\rho(s_i)<\al$, $\mrm{EXH}(\mc{A}\clrest H_i)=\mrm{FIN}(\mc{A}\clrest H_i)$ for every $i$. Also, we know that $\mc{A}\clrest (H\setminus\{\0\})=$
\[ \mc{A}\clrest \big(H_1\cup H_2\cup\dots\big)=\big\{E_1\cup E_2\cup\dots:E_i\in\mc{A}\clrest H_i\,\text{ for every }\,i\big\}.\]

\begin{nnclaim}
Given pairwise disjoint nonempty countable sets $\Omega_i$, $i\in\NN$ and hereditary covers by finite sets, $\mc{E}_i$ of $\Omega_i$, let 
\[ \mc{E}=\big\{E_1\cup E_2\cup\dots:E_i\in\mc{E}_i\,\text{ for every }\,i\big\}.\]
If $X_{\mc{E}_i}$ does not contain $c_0$ for any $i$, then nor does $X_\mc{E}$.
\end{nnclaim}
\begin{proof}[Proof of the Claim] Let $\Omega=\bigcup_{i=1}^\infty\Omega_i$, $x\in\mrm{FIN}(\mc{E})$, and fix an $\eps>0$. As $\|x\|_\mc{E}=\sum_{i=1}^\infty\|P_{\Omega_i}(x)\|_{\mc{E}_i}<\infty$, we can pick an $n$ such that $\sum_{i=n}^\infty\|P_{\Omega_i}(x)\|<\eps/2$. Also, as $P_{\Omega_i}(x)\in\mrm{FIN}(\mc{E}_i)=\mrm{EXH}(\mc{E}_i)$, we can pick finite sets $B_i\subseteq\Omega_i$, $i=1,2,\dots,n$ such that $\|P_{\Omega_i\setminus B_i}(x)\|<\eps/2n$. Now, if $B=\bigcup_{i=1}^nB_i$ then $\|P_{\Omega\setminus B}(x)\|= \sum_{i=1}^n \|P_{\Omega_i\setminus B_i}(x)\|+\sum_{i=n}^\infty\|P_{\Omega_i}(x)\|<\eps$. It follows that $x\in\mrm{EXH}(\mc{E})$.
\end{proof}

Applying this Claim to $\mc{E}_i=\mc{A}\clrest H_i$, we obtain that $X_{\mc{A}\upharpoonright (H\setminus \{\0\})}$ does not contain $c_0$, hence nor does $X_{\mc{A}\upharpoonright H}$, i.e. $H\in\mc{H}(\mc{A},\neg c_0)$. 

\smallskip
Regarding complexity, fix an embedding $e:\NN^{<\NN}\to 2^{<\NN}$ (w.r.t. extensions of nodes), then $\{T\subseteq \NN^{<\NN}:T$ is a tree on $\NN\}\to\mc{P}(2^{<\NN})$, $T\mapsto e[T]$ is a continuous reduction of the $\mbf{\Pi}^1_1$-complete set of well-founded trees to $\mc{H}(\mc{A},\neg c_0)$.
\end{proof}

At the other end, regarding $\mc{H}(\mc{A},\neg\ell_1)$, we have the following result, fundamentally due to Kunen (see \cite[Lemma 2.1]{howtodrive}). 
\begin{prop}\label{A-no-l1}
\begin{align*}
\mc{H}(\mc{A},\neg\ell_1)&=\big\{H\subseteq 2^{<\NN}:H\,\text{ does not contain infinite antichains}\,\big\}\\
&=\big\{H\subseteq 2^{<\NN}:H\,\text{ can be covered by finitely many chains}\big\}\\
&=\big\{H\subseteq 2^{<\NN}:\sup\big\{|A|:A\subseteq H\,\text{ is an antichain}\big\}<\infty\big\}
\end{align*} 
is the $F_\sigma$ ideal generated by all branches in $2^{<\NN}$. 
\end{prop} 
\begin{proof}

Let us denote the families on the right by $\mc{H}_1$, $\mc{H}_2$, and $\mc{H}_3$. Obviosuly, $\mc{H}_2$ is the ideal generated by all branches and $\mc{H}_3$ is $F_\sigma$. Also, $\mc{H}_2\subseteq\mc{H}_3\subseteq\mc{H}_1$ clearly holds. We need to show that $\mc{H}(\mc{A},\neg\ell_1)\subseteq \mc{H}_1\subseteq\mc{H}_2\subseteq\mc{H}(\mc{A},\neg\ell_1)$. The first inclusion is trivial. To see the last one, if $H$ can be covered by $n$ many chains, then $\mc{A}\clrest H\subseteq [H]^{\leq n}$, therefore, $(e_n)_{n\in H}$ is equivalent to the basis of $c_0$.

\smallskip
$\mc{H}_1\subseteq\mc{H}_2$: Assuming that $H$ cannot be covered by finitely many chains, we show that $H$ contains an infinite antichain. 
Let us consider $H$ a tree again (adding $\0$ if necessary). We can assume that its levels are finite. By recursion we can pick $\0=h_1\subseteq h_2\subseteq \dots$ from $H$, $h_n$ is from the $n$th level such that $\{t\in H:t$ extends $h_n\}$ cannot be covered by finitely many chains. Now consider the chain $\{h_n:n=1,2,\dots\}$, then for each $n$ we can pick a $t_n\in H$ extending $h_n$ not covered by this chain. It follows that $\{t_n:n=1,2,\dots\}\subseteq H$ is an infinite antichain.
\end{proof}

If $\mc{C}=\{F\subseteq 2^{<\NN}:F$ is a finite chain$\}$ from Example \ref{exa-chains}, then, applying Proposition \ref{A-no-l1}, it follows easily that 
\[ \mc{H}(\mc{C},\neg c_0)=\mc{H}(\mc{A},\neg\ell_1)=\big\{H\subseteq 2^{<\NN}:X_{\mc{C}\upharpoonright H}\simeq\ell_1(H)\big\}.\] 
Also, $\mc{H}(\mc{C},\neg\ell_1)=\mc{H}(\mc{A},\neg c_0)$ follows from Proposition \ref{A-no-c0} and the fact that if there are copies of $\ell_1$ in $X_\mc{F}$, then there are trivial ones. 

\section{Further questions}\label{quesec}

Consider $\mc{F}\in\mrm{FHC}$. Obviously, $\mrm{conv}(W(\max(\overline{\mc{F}})))\supseteq W(\overline{\mc{F}})$, therefore, regarding convex and linear hulls of $\mrm{Ext}(B(X^*_\mc{F}))=W(\max(\overline{\mc{F}}))$, we can work with $W(\overline{\mc{F}})$ instead. The first question is motivated by the fact (see \cite[Corollary 4.6]{ExtremeKevin}) that, assuming $\mc{F}$ is compact, $X^*_\mc{F}$ satisfies the $\lam$-property (i.e. the CSRP). 

\begin{que} 
Does $\overline{\mrm{span}}(W(\overline{\mc{F}}))=X^*_\mc{F}$, or even $\overline{\mrm{conv}}(W(\overline{\mc{F}}))=B(X^*_\mc{F})$, always hold? If so, then we may go further: Does $X^*_\mc{F}$ always satisfy the $\lam$-property?  
\end{que} 

In general, $c_0$-saturation of a Banach space does not imply $\ell_1$-saturation of its dual (see \cite{Leung}). However, the question is still open for the case of combinatorial Banach spaces, that is, for combinatorial spaces generated by
compact families. Note that every compact family $\mc{F}\in\mrm{FHC}$ admits a quasinorm $\|\bullet\|^\mc{F}$ on $c_{00}$ with the completion of $(c_{00},\|\bullet\|^\mc{F})$ being $\ell_1$-saturated and generating the dual space $X_\mc{F}^*$ as its Banach envelope (\cite{quasiJachimek}). 

\begin{que}
Assume $X_\mc{F}$ is $c_0$-saturated, i.e. $\mc{F}$ is compact. Is $X_\mc{F}^*$ $\ell_1$-saturated? 
\end{que}

There are obvious questions regarding missing implications from Section \ref{towards}, also, there is a very natural variant of ($\neg c_0$) we have not mentioned yet. 

\begin{que} 
Assume that $X_\mc{F}$ does not contain $c_0$, i.e. ($\neg c_0$) holds:
\[\tag{$\neg c_0$} \forall\text{ $\mc{F}$-supp. norm. bl. basic }\,(x_n)\,\text{ in }\,X_\mc{F}\;\sup_{A\in\overline{\mc{F}}}\sum_{n=1}^\infty\|P_A(x_n)\|_\mc{F}=\infty.\]
Does this imply that the formally stronger 
\[\tag{$\Sigma$} \forall\text{ $\mc{F}$-supp. norm. bl. basic }\,(x_n)\,\text{ in }\,X_\mc{F}\;\exists\,A\in\overline{\mc{F}}\;\,\sum_{n=1}^\infty\|P_A(x_n)\|_1=\infty\]
also hold? What can we say about its uniform version, 
\[\tag{$U_\Sigma$} \forall\text{ pairwise disjoint }(F_n)\text{ in }\mc{F}\setminus\{\0\}\;\exists\,A\in\overline{\mc{F}}\;\,\sum_{n=1}^\infty\frac{|A\cap F_n|}{|F_n|}=\infty,\]
e.g. does ($U_{\neg c_0}$) imply ($U_\Sigma$)?
\end{que} 

\begin{que} In the realm of combinatorial spaces,
(a) does the Schur property (i.e. ($S$)) imply ($S^*$), or (b) does $(U_{\neg c_0})$ imply nowhere compactness?
\end{que}

We already mentioned that everywhere perfectness feels odd mostly because it is not invariant under the equivalence of families (that is, of the generated norms). We know that our properties do not imply everywhere perfectness, not even ($S^*$), the
strongest among them. But one may wonder if the following is true:

\begin{que}
Assume $\mc{F}$ satisfies ($S^*$). Is $\mc{F}$ equivalent to an everywhere perfect family?
\end{que}

Also, the family $\mrm{iFh}(\mc{Q})$ in Subsection \ref{farah-families-1} may suggest that although the spaces induced by everywhere perfect families may not be $\ell_1$-saturated, perhaps they can be decomposed into an '$\ell_1$-saturated part' and
a '$c_0$-saturated part' - on the level of the family or on the level of the space. E.g. we may ask the
following question.

\begin{que} Suppose that $\mc{F}$ is an everywhere perfect family. Do there exist an $\ell_1$-saturated Banach space $X$ and a $c_0$-saturated Banach space $Y$ such that $X_\mc{F}$ is isomorphic to $X \oplus Y$?
\end{que}

We finish with questions concerning the ideals considered in Section \ref{J ideals}.

\begin{que} It is easy to see that if $H\subseteq 2^{<\NN}$ then the minimal number of antichains covering $H$ is equal to the supremum of the lengths of chains through $H$. (This is basically Mirsky's theorem / dual Dilworth theorem for $H\subseteq 2^{<\NN}$.) It follows that
\[ \mc{H}=\big\{H\subseteq 2^{<\NN}:H\,\text{ can be covered by finitely many antichains}\big\}\] is an $F_\sigma$ ideal, and it is strictly smaller than $\mc{H}(\mc{A},\neg c_0)$. Is $\mc{H}$ of the form $\mc{H}(\mc{F},\neg c_0)$ for some $\mc{F}$? 
\end{que} 

\begin{que} Can we characterize which ideals exactly are of the form $\mc{H}(\mc{F},\neg c_0)$ (or $\mc{H}(\mc{F},\neg\ell_1)$, etc)?
\end{que}

\bibliographystyle{siam}
\bibliography{bib-norm}
\end{document}